\documentclass[11pt]{amsart}
\usepackage[utf8]{inputenc}
\usepackage{amsmath,amsfonts,amssymb,amsthm}
\usepackage[abbrev,lite,nobysame]{amsrefs}
\usepackage{mathrsfs}
\usepackage{geometry}
\usepackage{enumitem}
\usepackage[dvipsnames,svgnames]{xcolor}
\usepackage{hyperref}
\usepackage[utf8]{inputenc}
\usepackage[english]{babel}
\usepackage[symbol]{footmisc}
\usepackage{fancyhdr}
\usepackage[centercolon]{mathtools}
\mathtoolsset{showonlyrefs}
\usepackage{microtype}
\usepackage{lmodern}
\usepackage[only,llbracket,rrbracket]{stmaryrd}
\usepackage{tikz}
\usepackage{dsfont}

\hypersetup{
    colorlinks=true,
    linkcolor=blue,     
    urlcolor=blue,
    citecolor=black
}

\DeclareMathAlphabet{\mathup}{OT1}{\familydefault}{m}{n}
\newcommand{\dd}[1]{\mathop{}\!\mathup{d} #1}
\newcommand{\ddt}{\frac{\dd}{\dd t}}
\newcommand{\e}[1]{\mathop{}\!\mathup{e} #1}

\def\Xint#1{\mathchoice
{\XXint\displaystyle\textstyle{#1}}%
{\XXint\textstyle\scriptstyle{#1}}%
{\XXint\scriptstyle\scriptscriptstyle{#1}}%
{\XXint\scriptscriptstyle\scriptscriptstyle{#1}}%
\!\int}
\def\XXint#1#2#3{{\setbox0=\hbox{$#1{#2#3}{\int}$ }
\vcenter{\hbox{$#2#3$ }}\kern-.59\wd0}}
\def\avint{\Xint-}

\newcommand{\grad}{\nabla}

\newcommand{\laplace}{\Delta}

\newcommand{\N}{\mathbb{N}}
\newcommand{\Z}{\mathbb{Z}}

\newcommand{\R}{\mathbb{R}}

\newcommand{\T}{\mathbb{T}}

\newcommand{\Prob}{\mathbb{P}}
\newcommand{\EX}{\mathbb{E}}

\DeclareMathOperator{\diag}{diag}

\newcommand{\Ha}{\ensuremath{\mathcal{H}}}

\newcommand{\eps}{\varepsilon}
\newcommand{\la}{\langle}

\newcommand{\ra}{\rangle}

\newtheorem{proposition}{Proposition}[section]
\newtheorem{theorem}{Theorem}[section]
\newtheorem{lemma}{Lemma}[section]

\theoremstyle{definition}

\theoremstyle{remark}
\newtheorem{remark}{Remark}[section]

\DeclareMathAlphabet{\mathup}{OT1}{\familydefault}{m}{n}

\newcommand{\si}{\mathrm{s}}
\newcommand{\co}{\mathrm{c}}

\numberwithin{equation}{section}

\usepackage{comment}

\newcommand{\mel}{\MoveEqLeft}
\begin{document}

\title[Exponential mixing by random cellular flows]{Exponential mixing by random cellular flows}

\author[V. Navarro-Fernández]{Víctor Navarro-Fernández}
\address{(VNF) Department of Mathematics, Imperial College London, London, UK}
\email{v.navarro-fernandez@imperial.ac.uk}

\author[C. Seis]{Christian Seis}
\address{(CS) Institut für Analysis und Numerik, Universität Münster, Münster, Germany}
\email{seis@uni-muenster.de}

\date{\today}

\begin{abstract}
We study a passive scalar equation on the two-dimensional torus, where the advecting velocity field is given by a cellular flow with a randomly moving center. We prove that the passive scalar undergoes mixing at a deterministic exponential rate, independent of any underlying diffusivity. Furthermore, we show that the velocity field enhances dissipation and we establish sharp decay rates that, for large times, are deterministic and remain uniform in the diffusivity constant. Our approach is purely Eulerian and relies on a suitable modification of Villani’s hypocoercivity method, which incorporates a larger set of Hörmander commutators than Villani's original method.
\end{abstract}

\maketitle

\tableofcontents

\section{Introduction}

Mixing is a fundamental aspect of fluid flows, occurring both in natural phenomena and in industrial applications.
Examples include the transport of heat and pollutants in the atmosphere and in oceans, the distribution of freshwater in coastal waters near estuaries or of melting water near the polar ice caps, and mixing in chemical and pharmaceutical manufacturing. It is the transfer of information from larger to smaller length scales, driven by shear, and leads to the formation of fine-scale filamentary structures.

In laminar shear flows, whether linear or circular, the emerging filaments often exhibit a few characteristic length scales that decrease over time at a rather slow rate. In contrast, turbulent flows typically operate across a wide range of scales, such as through the interplay of interacting coherent vortex structures, making mixing significantly more effective. In this case, the average length scale decreases substantially faster.

If the observable undergoes an independent diffusion process (or, more accurately, conduction in the case of heat), the rate of diffusion is typically enhanced by the underlying fluid flow. This results from the rapid dispersal of filaments due to microscopic stochasticity.
In contrast to the mixing process, which reduces scale through filamentation, diffusion reduces intensity gradients.
The length scale at which diffusion prevents the fluid flow from generating finer structures—because they are immediately smoothed out—is known as the Batchelor scale. It determines the rate at which the observable decays at large times, which is significantly larger than the purely diffusive dissipation rate. 

In the mathematics community, mixing and enhanced dissipation in fluid dynamics have been a major focus of research over the past two decades. In the simplest setting,  the observable is considered passive, which means that it gives no feedback to the advecting fluid flow.  Treating the purely advective and the diffusive setting simultaneously, the mathematical model describing the transport of an observable  $\theta = \theta(t,x)\in \R$ is the \emph{passive scalar equation}
\begin{equation}\label{eq:AD}
\partial_t \theta + v\cdot \grad \theta = \kappa \laplace\theta.    
\end{equation}
We consider this model on the two-dimensional flat torus $\T^2=[0,2\pi]^2$. Here, $v = v(t,x)\in\R^2$ is the vector field modelling the velocity of an incompressible fluid, and thus
\[
v(t,x) = -\grad^{\perp} \psi(t,x) = \begin{pmatrix}
    \partial_{x_2} \psi\\-\partial_{x_1} \psi
\end{pmatrix}(t,x),
\]
for some  stream function $\psi(t,x)$. We normalise the mean estrophy to unity, $
\|\grad v\|_{L^2(\T^2)} = 2\pi $, so that the non-negative diffusivity constant $\kappa$ in \eqref{eq:AD} plays the role of an inverse P\'eclet number, measuring the strength of diffusivity relative  to the strength of advection.  

A popular toy model featuring a simple   yet rich structure is the cellular flow generated by the stream function $\psi_c(x_1,x_2)   = \sin(x_1)\sin(x_2)$, that is,
\begin{equation}
    \label{32}
 v_c( x) = \begin{pmatrix}
    \sin( x_1)\cos( x_2)\\ -\cos( x_1)\sin( x_2)
\end{pmatrix},
\end{equation}
see Figure \ref{fig1}(a) for a visualisation. This model describes an array of vortices with alternating rotation, as has been studied both in experiments, for instance, in magnetically forced convection \cite{RothsteinHenryGollub99}, and in numerics \cite{MathewMezicGrivopoulosVaidyaPetzold07}. In the mathematical community, cellular flows play an important role in the theory of homogenisation in diffusive processes  \cites{FannjiangPapanicolaou94,Heinze03,Koralov04}. Their influence on mixing and enhanced dissipation was investigated, for instance, in \cites{YaoZlatos17,IyerZhou23,BrueCotiZelatiMarconi24}: In the autonomous setting \eqref{32},  the stream lines are closed, and thus, particles are transported along periodic orbits about the cells' centres. This rigidity confines the mixing efficiency and mixing rates are at best algebraic. See Figure \ref{fig2}(a) for a snapshot in a mixing process. 
\begin{figure}
    \centering
  (a)  \includegraphics[width=0.4\textwidth]{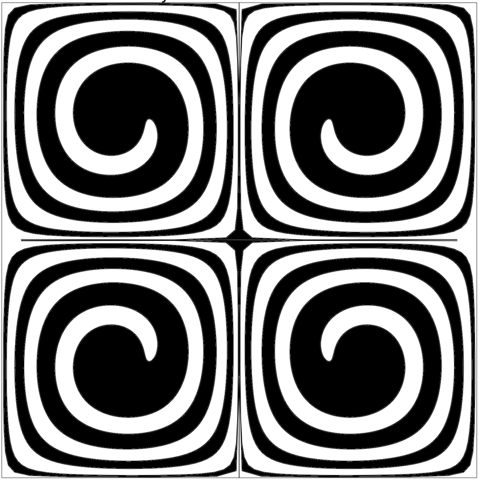}
     (b)   \includegraphics[width=0.4\textwidth]{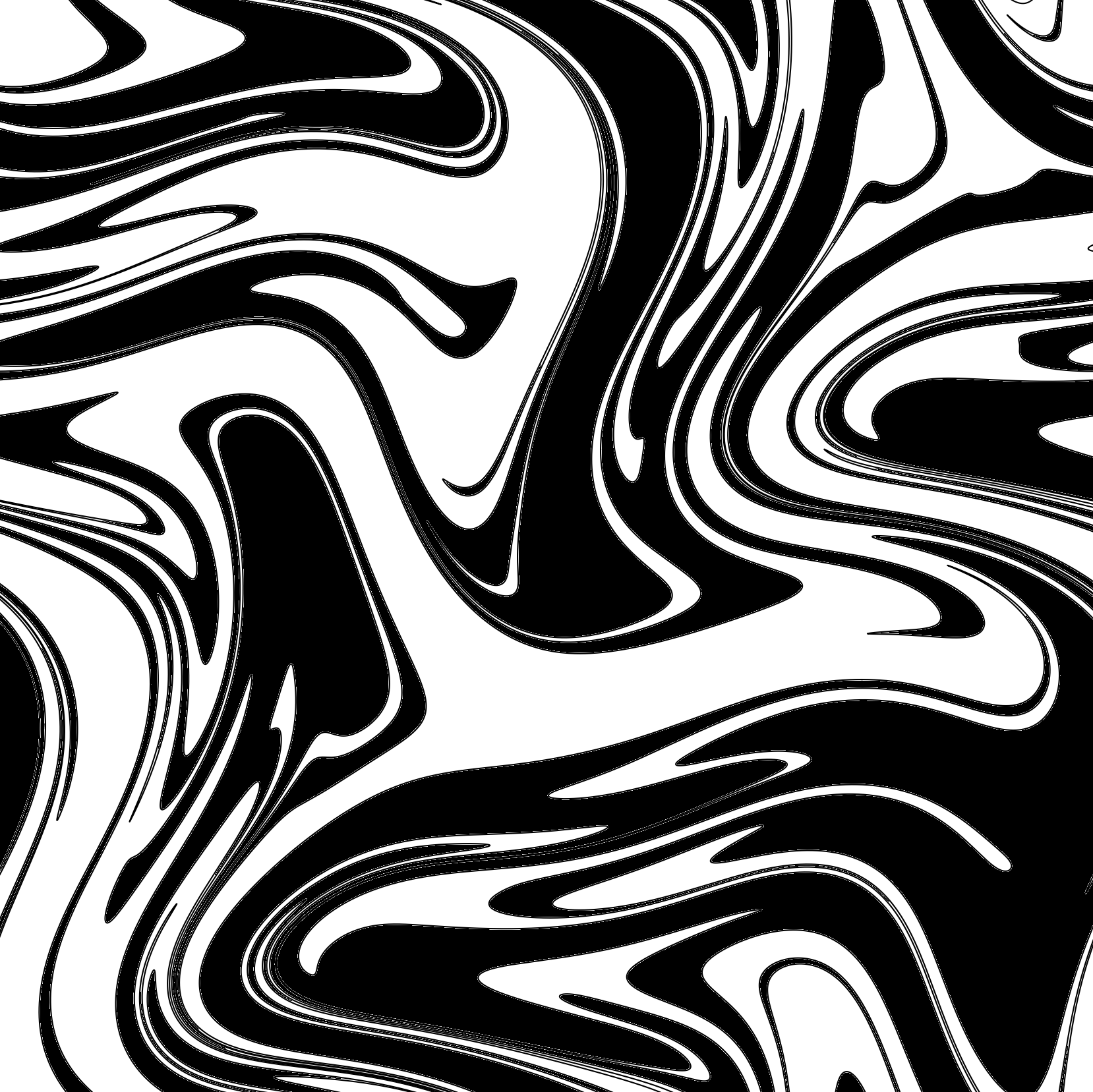}
    \caption{Snapshots of a simulation of the non-diffusive mixing process ($\kappa=0$) by (a) the steady cellular flow \eqref{32} and (b) the cellular flow with random phase shift \eqref{66}. Both numerical experiments start with the same circular and $\pm1$-valued  mean-zero initial configuration localized at the centre of the cell.}
    \label{fig2}
\end{figure}

In the present paper, we will perturb the steady cellular flow model by  randomly shifting the cell grid,
\begin{equation}\label{66}
v(t,x) = v_c(x-Y_{\nu}(t)),
\end{equation}
where $Y_{\nu}(t)$ is a stochastic process on the torus $\T^2$ driven by a standard Brownian motion at diffusivity $\nu$, see \eqref{heat} below.
This random cellular flow serves as a toy model for the interaction of vortex structures in a turbulent flow. See Figure \ref{fig2}(b) for a snapshot of the mixing process.  This model is reminiscent of the Pierrehumbert model \cite{Pierrehumbert1994}, where alternating shear flows with a random phase shift are considered. However, unlike the velocity in the Pierrehumbert model, ours varies continuously in time and is genuinely two-dimensional at any time. More specifically, instead of alternating shears in two orthogonal directions, our models superposes these. This point of view becomes more evident after a rotation of the coordinate system, see \eqref{31} below.

Because the cellular flow $v_c$ is a steady solution to the Euler equations, it follows from a formal application of the  chain rule that the random vector field \eqref{66} yields a special solution to the  Euler equations with stochastic transport noise
\[
\partial_t \omega + \left(v+ \ddt Y_{\nu}\right)\cdot \grad \omega = 0.
\]
Here, $\omega = \grad\times  u$ is the scalar vorticity in the fluid.

The purpose of our work is to establish that the passive scalar dynamics \eqref{eq:AD} induced by the simple random flow \eqref{66} shows typical features of turbulent dynamics, such as exponential mixing rates and diffusivity-independent enhanced dissipation rates. 

We start by addressing the non-diffusive setting.

 \begin{theorem}[Case $\kappa=0$]\label{Tk=0}
There exists a constant $\nu_0>1$ such that for any $\nu\ge \nu_0$ the following  holds true: There exists a deterministic constant $\lambda>0$ and an almost surely finite random constant $C>0$ such that for any mean-zero initial configuration $\theta_0\in H^1(\T^2)$, the solution $\theta$ of the passive scalar equation \eqref{eq:AD} is almost surely exponentially mixed at rate $\lambda$, that is, 
\[
\|\theta(t)\|_{\dot H^{-1}} \le C \e^{-\lambda t} \|\theta_0\|_{\dot H^1},
\]
for any $t\ge0$. The random constant satisfies $\EX_{\nu}[C^q]<\infty$ for any positive exponent $q$.
\end{theorem}

In the statement of the theorem, the expectation $\EX_{\nu}$ is understood with respect to any realization of the random process $Y_{\nu}$, see Section \ref{s:strategy} for details. 

Considering homogeneous negative Sobolev norms for measuring the average degree of filamentation   was proposed and promoted in \cites{MathewMezicPetzold05,Thiffeault12}. Particularly popular among these norms is the $\dot H^{-1}$ norm, because it has a conveniently simple representation and it scales precisely like  length. Hence, the estimate in Theorem \ref{Tk=0} can indeed be interpreted as a bound on the decay of the average length scale in the system, saying that the system is mixing (at least) with an exponential rate. That mixing cannot occur at a faster rater was established in \cites{CrippaDeLellis08,LunasinLinNovikovMazzucatoDoering12,Seis13a,IyerKiselevXu14,Leger18}, which implies that the result in Theorem \ref{Tk=0} is actually sharp.

Choosing a positive Sobolev norm on the right-hand side is necessary for the pointwise result because the inviscid problem is time reversible. However, an interesting consequence of our analysis is that the constraint on the initial datum can indeed be relaxed to $L^2(\T^2)$ provided that the decay estimate is averaged over all realisations of the noise,  see \eqref{100} in Remark \ref{remark:batchelor} or Lemma \ref{prop:mixing-L2rhs}. Analogous  estimates  were obtained in \cite{CR24+}.  Noise-averaged negative regularity mixing estimates, where both the solution and the initial datum are considered in some negative Sobolev norm $\dot H^{-s}$, were established in  \cite{BFPS24+}.

Proving almost-sure exponential mixing for stochastic flow models was pioneered in the works of Bedrossian, Blumenthal and Punshon-Smith \cites{BBPS21,BBPS22a,BBPS22}, where, among other models, the passive scalar equation for the stochastically forced two-dimensional Navier--Stokes equations was considered. The main idea in these works is to regard the flow map associated with the stochastic vector field as a random dynamical system. Using Harris’ theorem, the multiplicative ergodic theorem, and Furstenberg’s criterion, one then establishes the ergodicity of certain Markov processes defined by the flow, as well as the positivity of the Lyapunov exponent.  The authors first obtained optimal rates of enhanced dissipation \cites{BBPS21}, then proved that the Lagrangian flow generated by their stochastic vector fields is \emph{chaotic}, namely they proved positivity of the top Lyapunov exponent \cites{BBPS22a}, and finally they established almost sure exponential decay of the correlations, which implies the sought exponential mixing result \cite{BBPS22}. 

The method has been later expanded, establishing that different constructions of random vector fields are (almost sure) exponential mixers. An extension to the stochastically forced Navier--Stokes equations with \emph{degenerate} noise has been obtained in \cite{CR24+}. One of the most relevant new examples is the aforementioned Pierrehumbert model, which is a smooth alternating shear flow in $\T^2$.  It has been showed to be an exponential mixer if  phases are randomised \cites{BlumenthalCotiZelatiGvalani23} or if the switching time is randomised \cites{Cooperman23}. Recently uniform in diffusivity mixing has been obtained as well, leading to optimal enhanced dissipation rates \cites{CIS24+}. Following a similar strategy, a randomised version of the Arnold--Beltrami--Childress (ABC) flows in $\T^3$ has been latterly shown to mix every $H^1$ initial configuration exponentially fast \cite{CotiZelatiNavarroFernandez}. Another relevant example of a stochastic vector field that has been proved to mix exponentially fast, also uniformly in diffusivity, is given by the Kraichnan model, which consists of a white-in-time incompressible Gaussian random field, see \cites{CZDG24,GessYaroslavtsev21,LTZ24+}.

Our approach for deriving the exponential decay estimate on mixing is significantly different from those  mentioned in the previous paragraphs, in that we stay within the Eulerian framework \eqref{eq:AD} instead of studying  Lagrangian variables. More specifically, we consider the expectation of the two-point correlation function, which solves a deterministic hypoelliptic equation on $\T^6$. To obtain exponential decay  of that quantity, we follow Villani's hypocoercivity method \cite{Villani09}, which relies on the study of iterated commutators in the spirit of H\"ormander \cite{Hoermander67}. Villani proposed a scheme to build a Lyapunov function that is equivalent to a first order Sobolev norm. Establishing a Gronwall inequality for that functional then gives exponential decay in norm.
It turns out that we have to  extend Villani's Lyapunov functional. In fact, in his original approach, Villani includes commutators iteratively by commutating with the advection term.  For us, it is crucial to implement  further Hörmander commutators generated by the noise term. As we will see, these new commutators complement Villani's commutators in an orthogonal fashion.  We will introduce all these operators in Section \ref{s:hypocoercivity}, where we also present and discuss our modified hypocoercivity approach in detail.

Hypocoercivity in the context of shear flows was studied earlier in \cites{BedrossianCotiZelati17,CotiZelati20} to obtain (almost) optimal estimates on rates of enhanced dissipation. However, with regard to the strength of the diffusivity of the underlying model, their regime is different from ours. While in \cites{BedrossianCotiZelati17,CotiZelati20}, hypocoercivity was studied for a hypoelliptic version of the passive scalar equation \eqref{eq:AD}, in which the inverse P\'eclet number $\kappa$ is small, we study hypocoercivity for the expectation, in which the hypoelliptic operator has its origin in the stochastic noise. As will be see, in order to establish our decay estimates, the stochastic diffusivity coefficient $\nu$ is necessarily large.

While our  approach yields less detailed dynamical information than the framework developed by Bedrossian, Blumenthal, and Punshon-Smith---such as the uniform positivity of the top Lyapunov exponent---, it suffices for establishing global estimates like bounds on mixing rates, which is the main goal of the present paper. We  stress that our method relies on the fact that our flow is finite dimensional: it is completely described by a single velocity field that is shifted in time. It would be interesting to see whether the strategy also applies to other finite-dimensional flow models in which the velocity field is not explicit.

Working in an Eulerian framework, it is actually not difficult to   extend our mixing estimates from Theorem \ref{Tk=0} to the diffusive setting $\kappa>0$.  Moreover, since our mixing rate $\lambda$ remains independent of the inverse P\'eclet number $\kappa$, our mixing result easily implies enhanced dissipation.

\begin{theorem}[Case $\kappa\in(0,1)$]
\label{Tk>0}
There exists a constant $\nu_0>1$ such that for any $\nu\ge \nu_0$ the following is true: There exists a deterministic constant $\lambda>0$ and an almost surely finite random constant $C>0$, both independent of $\kappa$ and an almost surely finite random constant $\mu(\kappa)>0$ such that for any mean-zero initial configuration $\theta_0\in H^1(\T^2)$, the solution $\theta$ of the passive scalar equation \eqref{eq:AD} is almost surely exponentially mixed at rate $\lambda$, that is,
\[
\|\theta(t)\|_{\dot H^{-1}} \le C e^{-\lambda t}\|\theta_0\|_{\dot H^1},
\]
for any $t\ge0$. Moreover, the model is dissipation enhancing and it holds
\begin{equation}
    \label{67}
\|\theta(t)\|_{L^2} \le \min \left\{ \frac{C}{\sqrt{\kappa}} \e^{-\lambda t},\e^{-\mu(\kappa)t}\right\} \|\theta_0\|_{L^2},
\end{equation}
for any $t\ge0$. The random constants satisfy $\EX_{\nu}[C^q]<\infty$ and $\EX_{\nu}[\e^{q/\mu(\kappa)}]<\infty$ for any positive exponent $q$, and there is the asymptotic formula
\begin{equation}
    \label{68}
\lim_{\kappa\to0} \log\left(\frac1{\kappa}\right)\mu(\kappa) \in\R,
\end{equation}
almost surely.
\end{theorem}

In order to stress the significance of the   decay estimates on the intensity \eqref{67}, we notice that in the absence of advection, the rate $\mu(\kappa)\sim \kappa$ is optimal, which is much smaller than $\mu(\kappa)\sim \log^{-1}(1/\kappa)$ or $\lambda\sim 1$ in the regime of large P\'eclet numbers, $\kappa\ll1$. In this regime, our  estimate \eqref{67} includes two cases: First, for relatively small times, $t\lesssim \log\frac1{\kappa}$, we have a non-uniform rate of enhanced dissipation, which can be stated as an estimate without an extra  constant,
\[
\|\theta(t)\|_{L^2} \le \e^{-\mu(\kappa)t}\|\theta_0\|_{L^2}.
\]
The asymptotic rate for $\mu(\kappa)$ given in \eqref{68} was proved to be optimal in the  setting of Sobolev regular velocity fields \cites{Seis23,MeyerSeis24}. It implies that this exponent is vanishing in the limit $\kappa\to0$, which is expected if there is no further constant in this estimate, because the $L^2$ norm is preserved by the dynamics. For larger times, $t\gtrsim \log \frac1{\kappa}$, the exponential rate of enhanced dissipation becomes uniform in $\kappa$, 
\[
\|\theta(t)\|_{L^2} \le  \frac{C}{\sqrt{\kappa}}  \e^{-\lambda t}\|\theta_0\|_{L^2},
\]
with $\lambda\sim 1$.  We remark that by interpolating with the standard $L^2$ energy balance, the power on the prefactor can be made arbitrarily small at the expense of decreasing simultaneously the exponential rate of decay. We notice furthermore that the uniform exponential decay rate is consistent with the Batchelor scale being of the order $\sqrt{\kappa}$, cf.~\cite{Batchelor1959}.

Estimates on enhanced dissipation were inferred earlier from mixing rates on the non-diffusive model \cites{CotiZelatiDelgandioElgindi20,FengIyer19}. Enhanced dissipation for stochastic fluid models was established  also in \cites{BBPS21,BBPS22}. In the deterministic setting, there is quite a number of examples for shear flows, in which the rate is significantly smaller, $\mu(\kappa)\sim \kappa^p$ for some $p$ dependent on the degeneracies of the shears, see, e.g.~\cites{BedrossianCotiZelati17,CotiZelati20,FengMazzucatoNobili23,BrueCotiZelatiMarconi24}.
To the best of our knowledge, our model is the only one besides those studied in \cites{BBPS21,BBPS22,CIS24+}, for which mixing in the diffusive setting could be proved. Optimal lower bounds on mixing in the diffusive setting are still unknown, although there are available double exponential estimates derived first by Poon \cites{Poon96,MilesDoering18}, and a recent result about lower bounds for the specific case of stochastically forced two-dimensional Navier--Stokes equations \cites{HPSRY24+}.

\begin{remark}\label{remark:batchelor}
    A small modification of the proofs leading to the mixing estimates in Theorems \ref{Tk=0} and \ref{Tk>0} yields as well an averaged mixing result for $L^2$ initial data,
    \begin{equation}
\label{100}
    \EX_\nu \|\theta(t)\|_{\dot H^{-1}} \le C\|\theta_0\|_{L^2} \e^{-\lambda t}
    \end{equation}
for all $t>0$, see Lemma \ref{prop:mixing-L2rhs} in Section \ref{s:mixing} for a rigorous statement. This estimate has an interesting consequence, as established by Cooperman and Rowan in \cites{CR25+}: it provides bounds on the Batchelor spectrum, which describes the energy distribution of a passive scalar in a turbulent flow with small diffusivity $\kappa$. Batchelor predicts a $|k|^{-1}$ decay of the energy spectrum in the advective subrange $ |k|\lesssim \kappa^{-1/2}$. Theorem 1.5 in \cite{CR25+} is applicable to our random vector field $v$ and shows that, upon adding a stochastic source $g$ to the advection-diffusion equation \eqref{eq:AD}, the energy spectrum of the passive scalar $\theta_g$ summed over constant width annuli in Fourier space satisfies the bounds
\[
\frac{h}{Cr} \le \lim_{t\to\infty} \EX_\nu \sum_{r\leq|k|\leq r+h} |\widehat\theta_g(t,k)|^2 \le \frac{Ch}{r},
\]
for some $C>1$, where $1\lesssim r\lesssim \kappa^{-1/2}$, and where $\widehat\theta_g$ denotes the Fourier transform of the passive scalar. 
\end{remark}

We will prove Theorems \ref{Tk=0} and \ref{Tk>0} simultaneously. Apparently, if $\kappa$ is positive, the equation for the expectation loses its hypoelliptic nature and becomes elliptic. Nevertheless, to obtain mixing rates that remain uniform in $\kappa$, it is still crucial to apply Villani's hypocoercivity method. The Lyapunov function is then modified only by the inclusion of a single term, whose sole purpose is to balance unavoidable diffusive error terms.

As indicated above, for  the passive scalar equation \eqref{eq:AD},  considering the standard cellular flow \eqref{32} is equivalent to studying the tilted cellular flow 
\begin{equation}
    \label{31}
  \tilde v_c(x) =  \begin{pmatrix}
    \sin( x_2)\\\sin( x_1)
\end{pmatrix}.
\end{equation}
Indeed, both are related via the spiral  transformation $ v_c( x) = d^{-1} R^T  \tilde v_c(dR  x)$, where $d=\sqrt2$ defines a dilating factor and   
\[
R = \frac1{\sqrt2}\begin{pmatrix}
    1&-1\\1&1
\end{pmatrix}
\]
is a rotation matrix. Up to a rescaling of the diffusivity constant by a factor of $d^2=2$, this spiral transformation leaves the passive scalar equation \eqref{eq:AD} invariant.   Both vector fields are plotted in Figure \ref{fig1}. 
\begin{figure}
    \centering
  (a)  \includegraphics[width=0.45\textwidth]{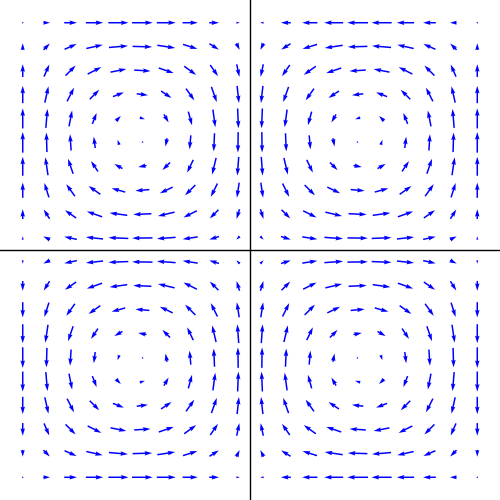}
     (b)   \includegraphics[width=0.45\textwidth]{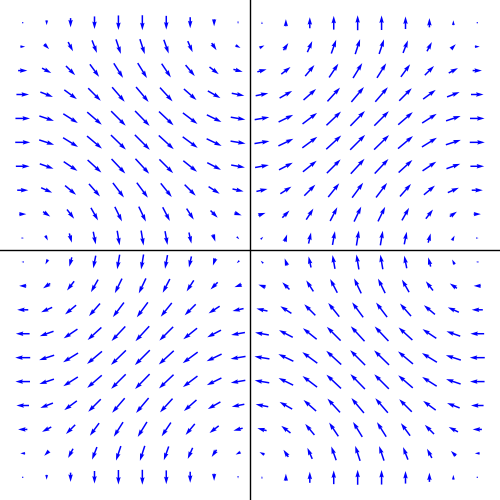}
    \caption{The two exemplary cellular vector fields $v_c$ given by \eqref{32} and $\tilde v_c$ given by \eqref{31} plotted on the flat torus $\T^2=[-\pi,\pi]^2$.}
    \label{fig1}
\end{figure}
For proving our Theorems \ref{Tk=0} and \ref{Tk>0}, it is thus enough to consider the slightly simpler vector field \eqref{31}, which we will do from here on. Moreover, we write by a slight abuse of notation,
\[
\tilde v_c(x,y) = \tilde v_c(x-y),
\]
so that $\tilde v(t,x) = \tilde v_c(x,Y_{\nu}(t))$. We will drop the tildes in the following.

Along this paper we will write $a\lesssim b$ if there exists a constant $C>0$, independent of all the relevant quantities involved on each estimate, such that $a\leq Cb$. Abusing slightly the notation, we will write $\|\cdot\|$ and $\langle\cdot,\cdot\rangle$ referring to the $L^2$ norms and inner products respectively, both on $\T^2$ and $\T^6$, although the domain of integration will always be clear from the context. When introducing any other norm $\|\cdot\|_X$ we will specify the corresponding Banach space $X$.

 \emph{This paper is organized as follows:} In Section \ref{s:strategy} we elaborate on the strategy that we follow to prove Theorems \ref{Tk=0} and \ref{Tk>0}. We explain how to obtain the decay of correlations and how to link it to the hypoelliptic PDE associated to the two-point process. In Section \ref{s:hypocoercivity} we prove that this PDE is indeed hypocoercive in a weighted $H^1$ norm that degenerates at the diagonal. In order to do so we adapt Villani's method to our setting by introducing a suitable energy functional, equivalent to the weighted $H^1$ norm, and demonstrate that it is exponentially decaying. Finally in Section \ref{s:mixing} we prove the decay of the correlations using the hypocoercivity of the two-point PDE obtained in the previous section. From the uniform-in-diffusivity decay of correlations result, we show how to obtain mixing and enhanced dissipation with an exponential rate independent of the diffusivity. In the Appendix \ref{s:feyman-kac} we elaborate on an application of the Feynman--Kac formula, that is a key element to obtain the mixing result.

\section{Strategy: Decay of correlations}\label{s:strategy}

The main goal of this paper is to obtain a mixing result with exponential rate for a class of cellular flows whose centre follows a random walk in $\T^2$. Even though our approach here is purely deterministic, we use some key ideas from previous works that put a focus on the study of the random dynamics generated by the flow associated to $v$, \cites{BBPS22,BBPS22a,BlumenthalCotiZelatiGvalani23}.

In our framework, the flow induced by the stochastic vector field is a volume-preserving mapping $X_t:\T^2\to\T^2$, $t\geq 0$, given by the SDE
\begin{equation}\label{flow}
\dd X_t = v_c(X_t,Y_t)\dd t + \sqrt{2\kappa}\dd B_t, \quad X_0 = \textrm{id}_{\T^2},
\end{equation}
where $B_t$ denotes a standard Brownian motion in $\T^2$, and $Y_t$ represents a stochastic process in $\T^2$  with diffusivity $\nu>0$ driven by another Brownian motion $\tilde B_t$,
\begin{equation}\label{heat}
\dd Y_t = \sqrt{2\nu}\dd \tilde B_t, \quad Y_0 = \textrm{id}_{\T^2}.
\end{equation}
In order to highlight the dependency of the flow on the stochastic process starting at some point $y\in\T^2$, we will occasionally write $X_t(x) = X_t(x;\{Y_s(y)\}_{s\in[0,t]})$. Notice that we can represent any solution to the PDE \eqref{eq:AD} using the flow map \eqref{flow} via the Feynman--Kac formula,
\[
\theta(t,x) = \EX_\kappa \theta_0(X_t^{-1}(x)),
\]
where $\EX_\kappa$ denotes the expectation operator with respect to the noise induced by the Brownian motion $B_t$. Of course, $\theta(t,x)$ is still a random quantity, in the sense that it still depends on the noise induced by the Brownian motion $\tilde B_t$. However, it does not depend on the particular choice of the Brownian motion: If $\hat B_t$ is another standard Brownian motion in $\T^2$ with diffusivity $\kappa$, and $Z_t$ is the solution to the associated SDE \eqref{flow} (where now $B_t$ is replaced by $\hat B_t$), then also $\theta(t,x) = \EX_{\kappa}\theta_0(Z_t^{-1}(x))$. In order to derive the PDE associated to the two-point process, it will be crucial to consider two independent (and thus, uncorrelated) Brownian motions.

Throughout the whole paper, and for the sake a better understanding, we will use the notation $\Prob_\kappa$ and $\EX_\kappa$ to denote the probability measure and expectation, respectively, associated to the noise induced by $B_t$ or $\hat B_t$, and $\Prob_\nu$ and $\EX_\nu$ to denote the probability measure and expectation associated to the noise induce by  $\tilde B_t$.

Based on the notion of \emph{strongly mixing mapping} from ergodic theory, we introduce the concept of the \emph{correlations} between any two functions $h$ and $g$ with zero mean,
\[
\text{Cor}_{X_t}(h,g) = \left| \int_{\T^2} g(x)\EX_\kappa h(X_t(x))\dd x \right|.
\]
Mixing, which shall be understood as a convergence to zero in the weak topology of $L^2$, is then related to the decay of this quantity: If $\textrm{Cor}_{X_t}(h,g)\to 0$ with probability $1$ and an exponential rate for any $h,g\in H^s(\T^2)$, then choosing $g=\theta_0$ and by a straightforward duality argument, we find that $\|\theta(t)\|_{H^{-s}}\to 0$ with probability $1$ and an exponential rate. Since $\textrm{Cor}_{X_t}(h,g)$ is a quantity that depends on the noise induced by $\tilde B_t$, at least formally we can compute the probability that at time $t=n\in \N$ this quantity is larger than $\e^{-\alpha n}$ for some $\alpha>0$ to be specified. Via Chebyshev inequality we find
\[
\begin{split}
\Prob_\nu [\textrm{Cor}_{X_n}(h,g)>\e^{-\alpha n}] & \leq \e^{2\alpha n}\EX_\nu\textrm{Cor}_{X_n}(h,g)^2 
\\
& = \e^{2\alpha n} \EX_\nu \iint_{\T^2\times\T^2} g(x)g(z)\EX_\kappa h(X_n(x))\EX_\kappa h(Z_n(z)) \dd x\dd z,
\end{split}
\]
because $\EX_{\kappa}h(X_n(z)) = \EX_{\kappa}h(Z_n(z))$. Here we introduce a new variable $z\in\T^2$ to rewrite the square of the integral using a double-variable approach. The information that we obtain from this little computation is that if we are able to prove that
\begin{equation}\label{Borel--Cantelli}
\EX_\nu \iint_{\T^2\times\T^2} g(x)g(z)\EX_\kappa h(X_n(x))\EX_\kappa h(Z_n(z)) \dd x\dd z \to 0
\end{equation}
as $n\to\infty$ with an exponential rate faster than $\e^{-2\alpha n}$, then using a Borel--Cantelli argument we would obtain that
\[
\Prob_\nu  \limsup_{n\to\infty} \left\lbrace\textrm{Cor}_{X_n}(h,g)>\e^{-\alpha n} \right\rbrace  = 0.
\]
This is a crucial step in the proof of exponential decay of correlations; for a detailed explanation we refer to  Section \ref{s:mixing}. The convergence of this quantity can be addressed using different techniques. In the original work by Bedrossian, Blumenthal and Punshon-Smith \cites{BBPS22,BBPS22a}, which was later extended to other different examples of vector fields \cites{BlumenthalCotiZelatiGvalani23,Cooperman23,CIS24+,CotiZelatiNavarroFernandez}, the authors use a random dynamical system framework to study the ergodic properties of the \emph{two-point Markov chain} that follows the trajectories of two particles starting from different points in the domain and driven by the same vector field. 

Our approach differs from this because we translate the study of the two-point Markov chain to the study of a two-point PDE, and then apply purely deterministic techniques to obtain the desired decay. Following the notation used in the previous computation, let $g=\theta_0$ and consider  a Lipschitz continuous probability density $\rho_0$ on $\T^2$, that is,
\begin{equation}\label{rho0}
\rho_0\in W^{1,\infty}(\T^2), \quad \rho_0\geq 0, \quad \int_{\T^2}\rho_0\dd y = 1,
\end{equation}
which models the initial distribution of the random shift process.
On a first note we observe that the function
\[
\rho(t,y) = \EX_\nu\rho_0(Y_t^{-1}(y))
\]
solves the heat equation
\[
\partial_t\rho = \nu\Delta_y\rho, \quad \rho(0,\cdot) = \rho_0 \quad \mbox{in }\T^2.
\]
Then, we define the function
\[
f(t,x,z,y) = \EX_\nu\left[ \EX_\kappa \theta_0(X_t^{-1}(x;\{Y_s(y)\}_{s\in[0,t]}))\EX_\kappa \theta_0(Z_t^{-1}(z;\{Y_s(y)\}_{s\in[0,t]}))\rho_0(Y_t^{-1}(y)) \right],
\]
which, via the  Feynman--Kac formula, is a solution to the so-called \emph{two-point PDE},
\[
\partial_tf + v_c(x,y)\cdot\nabla_xf + v_c(z,y)\cdot\nabla_zf = \kappa\left(\Delta_xf + \Delta_zf\right) + \nu\Delta_yf,
\]
with initial configuration $f_0(x,z,y) = \theta_0(x)\theta_0(z)\rho_0(y)$. This argument is elaborated with detail in Lemma \ref{lemma41} as well as in the Appendix \ref{s:feyman-kac}. It is clear that the initial datum $f_0$ is mean-zero with respect to $(x,z)$ because we choose $\theta_0$ mean-zero.

For the nondiffusive case, $\kappa=0$, the two-point PDE corresponds to a \emph{hypoelliptic} equation with transport in the directions $x$ and $z$ and with  diffusion only in the direction of the noise variable $y$. A suitable decay of the solutions to this PDE will imply the desired decay in \eqref{Borel--Cantelli}, and hence it will lead to the results stated in Theorems \ref{Tk=0} and \ref{Tk>0}.

Let us introduce some more convenient notation at this point, that will make the statements and proofs easier to read. We define $x=(x_1,x_2,x_3,x_4)\in\T^4$, $y=(y_1,y_2)\in\T^2$, and the expectation $f = f(t,x,y)$ for the two-point process. Then  two-point PDE becomes
\begin{equation}\label{hypo}
\partial_t f  + u\cdot \grad_x f = \kappa \laplace_x f +\nu \laplace_y f,   
\end{equation}
where now 
\begin{equation}\label{69}
u(x,y) = \begin{pmatrix}
    v_c(x_1-y_1,x_2-y_2)\\ v_c(x_3-y_1,x_4-y_2)
\end{pmatrix} =\begin{pmatrix}
    \sin(x_2-y_2)\\ \sin(x_1-y_1)\\ \sin(x_4-y_2)\\ \sin(x_3-y_1)
\end{pmatrix}.
\end{equation}

When studying the two-point problem, one of the most subtle aspects to consider is the behaviour of the dynamics near the diagonal $\mathbb{D} = \{x\in\T^4: x_1=x_3, x_2=x_4\}\subset \T^4$.  Indeed, if we examine the dynamics of two particles starting at the same position (i.e., on the diagonal), they will remain at the same point in space for all times due to the uniqueness of solutions for the transport equation \eqref{eq:AD}. As a result, two particles on the diagonal can never be separated, hence this implies that the diagonal can never be mixed. The diagonal itself is a null-set. However, a main difficulty to overcome is to ensure that particles are not getting trapped in a neighbourhood of it.

What we do to address the problem with the diagonal is to measure the decay of solutions to the two-point problem \eqref{hypo} in some weighted $H^1$ norm, that precisely degenerates in $\mathbb{D}$. 
More precisely, we consider the tailor-made norm
\begin{align*}
\|f\|_{\Ha^1 (\T^2\times \T^2\times \T^2)}  & = \|f\| +\|\grad_y f\|+  \|\partial_1 f+ \partial_3f \| +\|\partial_2f+\partial_4f\|\\
&\quad +\|\dd_{\mathbb{D}}(\partial_3f-\partial_1 f)\|+ \|\dd_{\mathbb{D}}(\partial_4f-\partial_2f)\|  + \sqrt{\kappa}\|\nabla_xf\|,
\end{align*}
where $\dd_{\mathbb{D}}$ is essentially the distance to the diagonal,
\[
\dd_{\mathbb{D}}(x) = \sqrt{\sin^2\left(\frac{x_3-x_1}2\right)+\sin^2\left(\frac{x_4-x_2}2\right)}.
\]
Here and in the following, the partial derivatives $\partial_i$ are always understood with respect to the $x$ variables, that is, $\partial_i = \partial_{x_i}$. 
The weights are vanishing precisely on the diagonal $\mathbb{D}$ at which the two-point process degenerates, and it is thus not surprising that we recover this degeneracy at the level of the PDE. Away from the diagonal, the terms that involve the $x$ derivatives yield a  full gradient control.

The desired decay of the $\Ha^1$ norm will be obtained via Villani's hypocoercivity method, see the monograph \cite{Villani09} for further information, where the definition of some suitable commutators will be the key to obtain decay in the $x$-variable.

The theorem reads a follows.

\begin{theorem}
    \label{T1} There exist $\nu_0>1$  such that given any $\nu_0<\nu$, there exist constants $\lambda=\lambda(\nu)>0$ and $C=C(\nu)\ge 1$ independent of $\kappa$, so that any solution $f$ to the two-point process problem \eqref{hypo}  with zero average in $x$ satisfies the exponential decay estimate
    \[
    \|f(t)\|_{\Ha^1(\T^2\times\T^2\times \T^2)}\le C \e^{-\lambda t}\|f_0\|_{\Ha^1(\T^2\times\T^2\times \T^2)},
    \]
    for any $t\ge 0$.
\end{theorem}

This estimate apparently implies that
\[
\|f(t)\|_{L^2(\T^2\times \T^2\times\T^2)} \le C \e^{-\lambda t} \|f_0\|_{H^1(\T^2\times \T^2\times\T^2)},
\]
for any $t\ge 0$, which will be enough for proving the announced mixing estimates for $\theta$.

Three remarks are in order.

\begin{remark}[Integrability of $\rho_0$]
As it will become apparent in Section \ref{s:mixing}, it is not really necessary that $\rho_0$ is a Lipschitz function in $\T^2$. Since $\rho(t,y) = \EX_\nu\rho_0(Y_t^{-1}(y))$ solves a heat equation, any $\rho_0$ Radon measure (including atomic measures) will become instantaneously smooth. Therefore, if we were to have initially a more singular probability measure $\rho_0$, we could start the process and consider a new initial datum $\rho(\tau,y)\in C^\infty\cap W^{1,\infty}$ in $\T^2$ for any $\tau>0$. We simply assume directly that $\rho_0$ is Lipschitz to avoid this technicality. Actually, due to the translation invariance of our periodic setting, it would be entirely enough to consider equally distributed starting points for the random shift process, that is, $\rho_0=(2\pi)^{-1}$.
\end{remark}

\begin{remark}[$\kappa$-dependence in Theorem \ref{T1}]\label{rmk:kappa-dep}
Theorem \ref{T1} is written in purpose with full generality for any $\kappa\geq 0$ to denote that even in the inviscid case $\kappa=0$, we get exponential decay of some $H^1$ norm, i.e., \eqref{hypo} is a hypocoercive PDE. If $\kappa>0$, then it is clear that due to coercivity, the $H^1$ norm will decay exponentially fast, but the exponential rate may depend on the value of $\kappa$. In Theorem \ref{T1} we can get rid of \emph{all} dependence of $\kappa\geq 0$: both the exponent $\lambda>0$ and also the constants $C$ are $\kappa$-independent.
\end{remark}

\begin{remark}[$\nu$-dependence in Theorem \ref{T1}]
Keeping track of the $\nu$-dependences in all estimates leading to Theorem \ref{T1}, it is possible to show that $\lambda \lesssim \nu^{-5}$ and $C\gtrsim \nu^6$. There are no reasons to believe that these bounds are optimal. 
\end{remark}

Theorem \ref{T1} is the technical ingredient that will imply the mixing and enhanced dissipation results that we are looking for. Without further ado, we introduce two propositions that show precisely how this implication works. On the one hand we have the mixing result.

\begin{proposition}\label{prop:mix}
If for any initial datum $f_0$ that is mean-zero in the $x$-variable the solution $f$ of equation \eqref{hypo} satisfies the estimate
\[
\|f(t)\|_{\Ha^1(\T^2\times\T^2\times \T^2)}\le C \e^{-\lambda t}\|f_0\|_{\Ha^1(\T^2\times\T^2\times \T^2)},
\]
with $\lambda>0$ and $C\ge 1$ independent of $\kappa$, then the solution $\theta$ of the passive scalar equation \eqref{eq:AD} is exponentially mixing uniformly in $\kappa$ and $\theta_0$ and with probability $1$.
\end{proposition}

On the other hand we write a proposition to show that uniform-in-diffusivity mixing implies enhanced dissipation with optimal rate. Even though this is a well-known result, we include it here for completeness and because our estimates are finer than those we are aware of in the literature, see \cites{BBPS21,CIS24+}.

\begin{proposition}\label{prop:diss}
Let    $\kappa\in(0,1)$ be fixed and suppose  there exist two constants $\lambda>0$ and $C>1$ independent of $\kappa$, such that any mean-zero solution  $\theta$ to the passive scalar equation \eqref{eq:AD}  satisfies the exponential mixing estimate
\[
\|\theta(t)\|_{\dot H^{-1}} \le C\|\theta_0\|_{\dot H^1}\e^{-\lambda t},
\]
for any $t\ge0$. Then the flow is dissipation enhancing: There exists two  constants  $\lambda'>0$ and $C'>1$   dependent only on $\lambda$  and $C$ such that
\[
\|\theta(t)\|_{L^2} \le   \frac{C'}{\sqrt{\kappa}}  \|\theta_0\|_{L^2} \e^{-\lambda' t},
\]
for all $t\geq 0$.  Moreover, there exists a constant $\mu(\kappa)>0$ dependent on both $C$ and $\lambda$ satisfying the asymptotic 
\[
\lim_{\kappa\to 0} \log\left(\frac1{\kappa}  \right)\mu(\kappa) = 2\lambda. 
\]
such that
\[
\|\theta(t)\|_{L^2} \le \e^{-\mu(\kappa)  t} \|\theta_0\|_{L^2},
\]
for any $t\ge 0$. Furthermore, if $\EX_{\nu}[C^q]<\infty$ for any $q>0$, then there also holds $\EX_{\nu}[\e^{q/\mu(\kappa)}]<\infty$ for any $q>0$.
\end{proposition}

Even if there is no reference to the dependence on the noise realisation stated in the proposition, it will be clear from the proof in Section \ref{s:mixing} that if the mixing estimate holds with probability $1$, then also the enhanced dissipation estimate does.

\section{The hypocoercivity estimate}\label{s:hypocoercivity}

Our main goal in the present section is the derivation of the hypocoercivity estimate that yields the exponential decay of the expectation, see Theorem \ref{T1}.

When studying this problem, it is convenient to perform the following spiral rotation of variables
\[
\tilde x_1 = \frac{ x_1+  x_3}2,\quad \tilde x_2 = \frac{x_2+x_4}2,\quad \tilde x_3= \frac{x_3-x_1}2,\quad \tilde x_4=\frac{x_4-x_2}2,
\]
and we write $\tilde y_1=-y_2$ and $\tilde y_2=-y_1$ because it will lead to more natural expressions in this analysis that follows. We then set
 \[
 \tilde f(\tilde x,\tilde y) = f(x,y) = f(\tilde x_1-\tilde x_3,\tilde x_2-\tilde x_4,\tilde x_3+\tilde x_1,\tilde x_4+\tilde x_2,-\tilde y_2,-\tilde y_1)
 .
 \]
  This change of variables leaves the $2\pi$-periodicity invariant. The transformed equation reads
\begin{equation}
    \label{18}
\partial_t \tilde f + \tilde u\cdot \grad_{\tilde x} \tilde f = \nu \Delta_y \tilde f + \tilde \kappa  \laplace_{\tilde x}\tilde  f,
\end{equation}
where $\tilde \kappa=\kappa/4$ and 
\begin{equation}
    \label{33}
\tilde u(\tilde x,y) = \begin{pmatrix}
\cos(\tilde x_4)\sin(\tilde x_2+\tilde y_1)\\
\cos(\tilde x_3)\sin(\tilde x_1+\tilde y_2)\\
\sin(\tilde x_4)\cos(\tilde x_2+\tilde y_1)\\
\sin(\tilde x_3)\cos(\tilde x_1+\tilde y_2)
\end{pmatrix} = \begin{pmatrix}
\co_4\si_2 \\ \co_3 \si_1\\ \si_4\co_2\\ \si_3\co_1
\end{pmatrix},
\end{equation}
where we use the short-hand notation $\si_1 = \sin(\tilde x_1 + \tilde y_2)$, $\si_2 = \sin(\tilde x_2 +\tilde  y_1)$, $\co_1 = \cos(\tilde x_1+\tilde y_2)$, $\co_2 = \cos(\tilde x_2+\tilde y_1)$, $\si_{3} = \sin(\tilde x_3)$, $\si_{4} = \sin(  \tilde x_4 )$, $\co_3 = \cos(\tilde x_3)$, $\co_4 = \cos(\tilde x_4)$. By this change of variables, the diagonal $x_1=x_3\cong x_3-2\pi$ and $x_2 = x_4\cong x_4-2\pi$ is mapped onto $\tilde x_3=0,\pi$ and $\tilde x_4 = 0,\pi$. The degeneracy at $\tilde x_3,\tilde x_4 = \pm\frac{\pi}2$ corresponds to points of maximal distance $x_3-x_1=x_4-x_2=\pi$, where the vector field in \eqref{69} degenerates because, for instance, $\sin(x_3-y_1) = -\sin(x_1-y_1)$.

For further reference, we notice that the $2\pi$-periodicity of $f$ in each variable and our particular change of variables implies that
\[
\tilde f(\tilde x_1+\pi,\tilde x_2,\tilde x_3+\pi,\tilde x_4) = \tilde f(\tilde x_1,\tilde x_2,\tilde x_3,\tilde x_4),
\]
Therefore, 
\begin{equation}
    \label{50}
\tilde F_3(\tilde x_1,\tilde x_2,\tilde x_3,\tilde x_4) := \frac12 \left(f(\tilde x_1 ,\tilde x_2,\tilde x_3,\tilde x_4) + f(\tilde x_1+\pi,\tilde x_2,\tilde x_3,\tilde x_4)\right),
\end{equation}
defines a function that is $\pi$-periodic in $\tilde x_3$. Analogously, we may define $\tilde F_4$ a periodic function in $\tilde x_4$.

We finally remark that under the above change of variables, the $\Ha^1$ norm gets  a bit handier
\[
\|\tilde f\|_{\Ha^1} = \|\tilde f\| + \|\tilde \partial_1\tilde f\| + \|\tilde \partial_2\tilde f\| + \|\widetilde{\mathrm{d}}_{\mathbb{D}}\tilde \partial_3 \tilde f\| + \|\widetilde{\mathrm{d}}_{\mathbb{D}}\tilde\partial_4 \tilde f\|  + \sqrt{\tilde \kappa}\|\grad_{ \tilde x}\tilde f\|+ \|\grad_{\tilde y} \tilde f\|,
\]
where now 
\[
\widetilde{\mathrm{d}}_{\mathbb{D}} = \sqrt{\si_3^2+\si_4^2}
\] 
is essentially the distance to the planes of degeneracies.
To simplify the notation, we will drop the tildes in the following. 

\medskip

We will obtain the estimate in Theorem \ref{T1} from the decay properties of a suitably modified Lyapunov functional in the spirit of Villani's hypocoercivity approach. The functional is generated by successively incorporating commutator terms $C_i$, which are iteratively derived by commutation with the advection operator $u\cdot\grad_x$,  starting with $C_0=\grad_y$. It is given by
\begin{equation}\label{energy}
\begin{split}
    \Phi(f) & = \frac{1}{2}\|f\|^2 + \frac{\alpha_0}{2}\|\nabla_yf\|^2 + \frac{\kappa\delta}{2}\|\nabla_xf\|^2 + \beta_0\langle \nabla_yf,C_1f\rangle + \frac{\alpha_1}{2}\|C_1f\|^2 + \frac{\gamma_1}{2}\|S_1f\|^2 \\
    & \quad +\beta_1\langle C_1f,C_2f\rangle +  \frac{\alpha_2}{2}\|C_2f\|^2 + \frac{\gamma_2}2 \|S_2f\|^2 + \beta_2 \langle C_2f,C_3f\rangle + \frac{\alpha_3}2\|\grad_x'f\|^2 .
\end{split}
\end{equation}
Here, we have written
\[
\grad_x' = (\partial_1,\partial_2)^T
\]
in order to simplify the notation.
The precise definition of our commutator terms $C_i$ and their \emph{orthogonal complements} $S_i$ (which are actually further commutators in the Hörmander sense), which will play a key role in our analysis and which are not included in Villani's original work, will be postponed to Subsection \ref{SS3}. Here, $\grad_x'$ is a simplification of the third commutator term $C_3$, which we will also introduce only later. First, we will   discuss the functional inequalities we have to establish.

The choice of coefficients in $\Phi(f)$ is subtle. Most of them are related via certain conditions that are required in order to control error terms. The compatibility of those has be checked eventually, which is the content of the subsequent lemma. The reader who is familiar with the hypocoercivity method will notice that most of the conditions are identical to those in Villani's original work. Those for $\gamma_i$ are new because the orthogonal complements $S_i$ were not considered so far in the literature.

For a better readability, we write $\alpha\ll_{\eps} \beta$ for $\alpha \le \eps \beta$ in the sequel. 

\begin{lemma}
    \label{L11}
For any $\eps\in(0,1)$, there exist positive constants $\nu$, $\alpha_0$, $\beta_0$, $\alpha_1$, $\gamma_1$, $\beta_1$, $\alpha_2$, $\gamma_2$, $\beta_2$, and $\alpha_3$ solving the conditions 
\begin{equation}\label{22}
\nu^5\beta_2\ll_{\eps}\nu^6\alpha_3\ll_{\eps}\nu^4\gamma_2\ll_{\eps}\nu^4 \alpha_2 \ll_{\eps}\nu^3 \beta_1\ll_{\eps}\nu^2\gamma_1\ll_{\eps}\nu^2 \alpha_1\ll_{\eps}\nu\beta_0\ll_{\eps}\alpha_0\ll_{\eps}1\ll_{\eps}\nu,
\end{equation}
and 
\begin{gather*}
\alpha_0^2\ll_{\eps}\nu\beta_0,\quad \beta_0^2 \ll_{\eps} \nu \beta_1, \\
\alpha_1^2\ll_{\eps}\beta_0\beta_1,\quad \alpha_2^2 \ll_{\eps}\beta_1\beta_2,\\
\beta_0^2 \ll_{\eps} \alpha_0\alpha_1,\quad \beta_1^2\ll_{\eps}\alpha_1\alpha_2,\quad \beta_2^2\ll_{\eps}\alpha_2 \alpha_3,\quad 
\beta_1^2\ll_{\eps}\beta_0\beta_2, \\
\nu \alpha_1^2\ll_{\eps}\beta_0\gamma_1,\quad 
\gamma_1^2\ll_{\eps} \nu^2 \alpha_1\alpha_2 ,\quad \gamma_2^2 \ll_{\eps}\nu \alpha_2 \beta_2.
\end{gather*}
The coefficients can be all chosen so that in \eqref{22} the scaling in $\nu$ is saturated, that is,
\begin{equation}
    \label{23}
\alpha_i\sim \gamma_i \sim \nu^{-2i},\quad \beta_i\sim \nu^{-2i-1}.
\end{equation}
\end{lemma}

We provide the proof of this  lemma in Subsection \ref{SS1} below.

Our hypocoercivity estimate in Theorem \ref{T1} will be a consequence of the corresponding Gronwall inequality for $\Phi(f)$. In order to establish that inequality, a major effort will be necessary in order to ensure that the functional which contains all individual terms that are dissipated by $\Phi(f)$ does also control the error terms that occur when differentiating the Lyapunov functional $\Phi(f)$ in time. This dissipation functional is given by
\begin{equation}\label{diss}
\begin{split}
    \Psi(f)  & = \nu \|\grad_y f\|^2 + \nu\alpha_0 \|\laplace_y f\|^2 + \beta_0 \|C_1f\|^2 + \nu \alpha_1 \|C_1\grad_y f\|^2 + \nu \gamma_1 \|S_1 \grad_y f\|^2 \\
    & \quad + \beta_1 \|C_2f\|^2 +\nu \alpha_2 \|C_2 \grad_y f\|^2 + \nu \gamma_2 \|S_2\grad_y f\|^2 + \beta_2 \|C_3 f\|^2 + \nu \alpha_3 \|\nabla_x'\grad_y f\|^2 \\
    & \quad + \kappa \left[ \|\nabla_xf\|^2 + \alpha_0\|\nabla_{x}\grad_yf\|^2 + \kappa\delta\|\laplace_x f\|^2 + \alpha_1\|C_1\grad_x f\|^2 + \gamma_1\|S_1\grad_x f\|^2 \right. \\
    & \qquad \left. + \alpha_2\|C_2\grad_x f\|^2 + \gamma_2\|S_2\grad_x f\|^2 + \alpha_3\|\grad_x'\grad_x f\|^2 \right].
\end{split}
\end{equation}

In the following, we will outline how we derive the bound of Theorem \ref{T1} using the hypocoercivity method. 
Our first main objective is to ensure that $\Psi(f)$ does indeed control all error terms that occur when differentiating $\Phi(f)$. 

\begin{proposition}\label{P1}
    Let $\eps\in(0,1)$ be given and choose   $\alpha_0$, $\beta_0$, $\alpha_1$, $\gamma_1$, $\beta_1$, $\alpha_2$, $\gamma_2$, $\beta_2$, and $\alpha_3$ as in Lemma \ref{L11} and $\delta=\eps$. If $\eps$ is sufficiently small, then $\Phi(f)$ satisfies the differential inequality
    \[
\frac{\dd}{\dd t} \Phi(f) + \frac12\Psi(f) \le 0.
\]
\end{proposition}

The proof of this proposition is the content of Subsection \ref{SS2}.

In order to turn this differential inequality into a Gronwall inequality, we need the following bound, at the heart of which lies a  Poincar\'e--Hardy type inequality.

\begin{lemma}
    \label{L12}
    For any  $\alpha_0$, $\beta_0$, $\alpha_1$, $\gamma_1$, $\beta_1$, $\alpha_2$, $\gamma_2$, $\beta_2$, and $\alpha_3$ fixed as in Lemma \ref{L11} and satisfying \eqref{23}, and $\delta=\eps$, there exists a constant $C>0$ such that
    \[
    \Phi(f)\le C\nu^5 \Psi(f),
    \]
    for any function $f$ for which $\Psi(f)$ is finite.
\end{lemma}
We refer again to Subsection \ref{SS1} for a proof of this lemma.

Assuming \eqref{23} is actually not a requirement for establishing a lower bound of $\Psi(f)$ in terms of $\Phi(f)$. However, it allows us  identifying the scaling in $\nu$, which we make explicit here but which we will not discuss further.

Combining the bound from Lemma \ref{L12} with the differential inequality in Proposition \ref{P1} yields exponential decay of our Lyapunov functional,
\[
\Phi(f(t))\le \e^{-\lambda t}\Phi(f_0),
\]
for any $t\ge0$. The small constant $\lambda$ is of the order $\nu^{-5}\ll1 $. For the statement of Theorem \ref{T1}, it now remains to establish the equivalence of the Lyapunov functional with the $\Ha^1$ Sobolev norm.

\begin{lemma}
    \label{L14}
   Let  $\alpha_0$, $\beta_0$, $\alpha_1$, $\gamma_1$, $\beta_1$, $\alpha_2$, $\gamma_2$, $\beta_2$, and $\alpha_3$ be fixed as in Lemma \ref{L11} and satisfying \eqref{23}. There exists a constant $C>0$ such that
    \[
   \frac1{C \nu^6}\|f\|_{\Ha^1}^2\le \Phi(f)\le C\|f\|_{\Ha^1}^2,
    \]
    for any $\Ha^1$ function $f$.
\end{lemma}

Again, we refer to  Subsection \ref{SS1} for the proof. First, we will introduce and discuss the commutator terms.

\subsection{The commutators}\label{SS3}

Starting with $C_0=\grad_y$, the first Villani commutator $C_1=[C_0,u\cdot\grad_x]$ is readily computed.

\begin{lemma}
    \label{L15}
    It holds that
    \[
    C_1 f = [\grad_y,u\cdot\grad_x]f =   \begin{pmatrix}
    \co_4\co_2\partial_1f -\si_4\si_2\partial_3 f \\
    \co_3\co_1\partial_2 f - \si_3\si_1 \partial_4 f
\end{pmatrix}.
    \]
\end{lemma}

\begin{proof}
    Because $\grad_y$ is independent of $x$, it is  clear that $C_1f=(\grad_yu)^T\grad_xf$, and the stated formula is immediately verified.
\end{proof}
We will see that  this term, and also the associated dissipation term $C_1\grad_y f$ are actually controlled during the evolution: It holds that $\beta_0\|C_1f\|^2+\nu\alpha_1 \|C_1\grad_y f\|^2\le \Psi(f)$.
It turns out that in addition to considering the first commutator, having its orthogonal complement will be beneficial for our analysis. It is given by
\begin{equation}
    \label{34}
S_1f = \begin{pmatrix}
   \co_4 \si_2\partial_1 f + \si_4\co_2\partial_3 f \\
   \co_3 \si_1\partial_2 f + \si_3\co_1\partial_4 f
\end{pmatrix}.
\end{equation}
The merit of including the orthogonal complement and a justification of this notion can be seen by the following pointwise identity.

\begin{lemma}\label{lemma:bound-D1}
We have
\[
|M_1 \grad_xf|^2 =  |C_1f|^2 + |S_1f|^2,
\]
where 
\begin{align*}
M_1 &= \mathrm{diag}\left(
    \co_4,\co_3,\si_4,\si_3\right).
    \end{align*}
\end{lemma}

\begin{proof}
Notice that we can write, using the Pythagorean identity,
\[
\begin{split}
(C_1f)_1^2 + (S_1f)_1^2 & =\co_4^2(\partial_1 f)^2 + \si_4^2 (\partial_3 f)^2,
\end{split}
\]
which immediately gives the desired result for the derivatives $\partial_1f$ and $\partial_3f$. The analogous statement for the remaining derivatives follows identically from the second components.
\end{proof}

The control of the first orthogonal complement by the dissipation term can be  established via a Hardy--Poincar\'e inequality.

\begin{lemma}
    \label{L9}
We have the estimate
\[
 \|S_1f\|^2 \lesssim \|C_1f\|^2 + \|C_1\nabla_yf\|^2.
\]
In particular, if the coefficients of $\Psi$ satisfy the relations of Lemma \ref{L11}, we control
\[
\nu \alpha_1 \|M_1 \grad_x  f\|^2 +\nu\alpha_1\|S_1f\|^2 \lesssim \Psi(f).
\]
\end{lemma}

\begin{proof}
    Notice that the first   commutator $C_1$ and its orthogonal complement $S_1$ are related by differentiation,
    \[
    (S_1f)_i = -(\partial_{y_i}C_1f)_i,\quad (\partial_{y_i}S_1f)_i=(C_1f)_i,
    \]
 and they are component-wise commuting, $[S_1,C_1]=0$, see also Lemma \ref{L21}. Therefore, we obtain via some integrations by parts 
\[
\begin{split}
    \|(S_1f)_i\|^2 
    & =  \langle (C_1f)_i,(\partial_{y_i}S_1f)_i\rangle + \langle (C_1\partial_{y_i}f)_i,(S_1f)_i\rangle + \langle (C_1f)_i,(S_1\partial_{y_i}f)_i\rangle \\
    & = \|(C_1f)_i\|^2 + 2\langle (C_1\partial_{y_i}f)_i,(S_1f)_i\rangle.
\end{split}
\]
We deduce the first statement after using Hölder's and Young's inequality  appropriately. The addendum then follows by inspection of the dissipation term via Lemma \ref{lemma:bound-D1}, provided that $\nu\alpha_1 \lesssim \beta_0$.
\end{proof}

Villani's hypocoercivity method always produces next-order commutators. The major task is to analyse which part of these is already controlled by previous commutators, hoping that at some point a higher-order commutator is completely controlled. We thus turn to the second-order commutators.

\begin{lemma}
    \label{L2}
    It holds that 
    \[
  [C_1,u\cdot\nabla_x]f =  C_2f +R_2f ,
    \]
    where 
    \[
C_2f = \ \begin{pmatrix}
 \si_3\co_1(\si_4\co_2\partial_1 f +\co_4\si_2 \partial_3 f)  \\ 
   \si_3\si_1(\si_4\si_2 \partial_1 f -\co_4\co_2\partial_3 f
\end{pmatrix}
+ \begin{pmatrix}
    \si_4\si_2(\si_3\si_1\partial_2 f -\co_3\co_1\partial_4 f)  \\
    \si_4\co_2 (\si_3\co_1\partial_2 f + \co_3 \si_1\partial_4 f)
\end{pmatrix},
\]
and, if the coefficients of $\Psi$ satisfy the relations in Lemma \ref{L11}, the remainder $R_2$ satisfies the estimate
\[
\nu \alpha_1 \|R_2 f\|^2  \lesssim  \Psi(f).
\]
\end{lemma}

\begin{proof}
A direct computation shows that
\[
\begin{split}
    [C_1,u\cdot\nabla_x]f & = \begin{pmatrix}
       \co_4\co_2(\co_3\co_1\partial_2f - \si_3\si_1\partial_4f) \\
       \co_3\co_1(\co_4\co_2 \partial_1 f  -\si_4\si_2\partial_3f)
    \end{pmatrix} 
    + \begin{pmatrix}
        \co_3\si_1(\co_4\si_2\partial_1 f + \si_4\co_2\partial_3f) \\ \co_4\si_2 (\co_3\si_1\partial_2 f + \si_3\co_1 \partial_4f)
    \end{pmatrix} \\
    & \quad + \begin{pmatrix}
   \si_4\si_2(\si_3\si_1\partial_2 f -\co_3\co_1\partial_4 f) \\ 
   \si_3\si_1(\si_4\si_2 \partial_1 f -\co_4\co_2\partial_3 f
\end{pmatrix}
+ \begin{pmatrix}
    \si_3\co_1(\si_4\co_2\partial_1 f +\co_4\si_2 \partial_3 f) \\
    \si_4\co_2 (\si_3\co_1\partial_2 f + \co_3 \si_1\partial_4 f)
\end{pmatrix}\\
    & = \begin{pmatrix}
   \co_4 \co_2(C_1f)_2 \\
    \co_3 \co_1(C_1f)_1
\end{pmatrix} + \begin{pmatrix}
   \co_3 \si_1(S_1f)_1 \\
  \co_4 \si_2(S_1f)_2
\end{pmatrix} + C_2f,
\end{split}
\]
so that the claim of the lemma follows from the definition of the dissipation term $\Psi(f)$, Lemma \ref{L9}, and the ordering of the coefficients in Lemma \ref{L11}.
\end{proof}

It is immediately clear that the new commutator term $C_2f$ cannot be controlled by  $M_1\grad_x f$, because the degeneracies are shifted by $\pi/2$.   However, similarly to the first order commutators, we obtain sharp estimates near the diagonal by including the second orthogonal complement, given by
 \[
S_2f = \begin{pmatrix}
\si_3\co_1(\si_4\si_2\partial_1 f - \co_4\co_2\partial_3 f)\\ \si_3\si_1(\si_4\co_2\partial_1 f + \co_4\si_2\partial_3f)
\end{pmatrix}
-  \begin{pmatrix}
    \si_4\co_2(\si_3\si_1 \partial_2 f - \co_3\co_1\partial_4 f)\\ \si_4\si_2(\si_3\co_1\partial_2 f +\co_3\si_1\partial_4 f)
\end{pmatrix}.
\] 
This operator is actually not perfectly orthogonal as it was the case for the first-order terms.

Let us first provide  point-wise   lower bounds.

\begin{lemma}\label{L1}
    We have \[
    |M_2\grad_x f|^2 \le |C_2 f|^2 + |S_2 f|^2 + 3|M_1\grad_x f|^2,
    \]
    where 
    \begin{align*}
M_2 &= \mathrm{diag}\left(
    \si_3 , \si_4,\si_3 ,\si_4 \right) .
    \end{align*}
\end{lemma}

\begin{proof}
    By using  the Pythagorean identity a couple of times, we write
    \begin{align*}
     |C_2 f|^2 +|S_2 f|^2
& = \si_3^2(\si_4\co_2\partial_1 f +\co_4\si_2 \partial_3f)^2 + \si_3^2 (\si_4\si_2 \partial_1 f -\co_4\co_2 \partial_3f)^2\\
&\quad + \si_4^2 (\si_3\si_1\partial_2 f - \co_3\co_1 \partial_4f)^2 + \si_4^2 (\si_3 \co_1 \partial_2 f +\co_3\si_1 \partial_4 f)^2\\
&\quad  - 2\si_3\co_3\si_4\co_4 \partial_3 f \partial_4 f\\
& = \si_3^2\si_4^2 (\partial_1f)^2 + \si_3^2\si_4^2 (\partial_2f)^2 + \si_3^2\co_4^2 (\partial_3 f)^2 + \co_3^2 \si_4^2 (\partial_4 f)^2\\
&\quad  - 2\si_3\co_3\si_4\co_4 \partial_3 f \partial_4 f.
\end{align*}
 Thanks to   Young's inequality, the mixed term can be   estimated as follows,
\begin{align*}
2\left|\si_3\co_3\si_4\co_4 \partial_3 f \partial_4 f \right|\le \si_4^2 (\partial_3f)^2 +\si_3^2 (\partial_4f)^2 \le |M_1\grad_x f|^2,
\end{align*}
and thus,
\[
\si_3^2\si_4^2 (\partial_1f)^2 + \si_3^2\si_4^2 (\partial_2f)^2 + \si_3^2\co_4^2 (\partial_3 f)^2 + \co_3^2 \si_4^2 (\partial_4 f)^2 \le  |C_2 f|^2 +|S_2 f|^2 + |M_1\grad_x f|^2.
\]
To improve the factors on the left side, it remains to notice using the Pythagorean identity that, for example
\[
\si_3^2(\partial_1 f)^2  =\si_3^2(\si_4^2+\co_4^2)(\partial_1 f)^2 \le \si_3^2\si_4^2 (\partial_1 f)^2 + \si_3^2 |M_1\grad_x f|^2.
\]
Arguing for the remaining terms analogously and invoking the Pythagorean identity once more, we derive the statement of the lemma.
\end{proof}

Now, we explain how the second orthogonal complement can be controlled by the dissipation term.

\begin{lemma}\label{lemma:S2-bound}
We have that
\[
\|S_2f\| \lesssim   \|C_2f\| + \|C_2\nabla_yf\|  + \|\grad_y f\| + \|M_1\grad_x f\|.
\]
In particular, if the coefficients of $\Psi$ satisfy the ordering of Lemma \ref{L11}, then 
\[
\nu \alpha_2 \|M_2\grad_x f\|^2+ \nu \alpha_2 \|S_2 f\|^2  \lesssim \Psi(f).
\]
\end{lemma}

\begin{proof}We consider the first entry of $S_2f$.
Using the identities $(S_2)_1 = -(\partial_{y_1}C_2)_1$ and $(\partial_{y_1} S_2)_1 = (C_2)_1$,  and some integrations by parts, we find
\[
\begin{split}
\|(S_2f)_1\|^2 &  = -\langle (S_2f)_1, (\partial_{y_1}C_2f)_1\rangle \\
    & = \langle (\partial_{y_1}S_2f)_1, (C_2f)_1\rangle + \langle (S_2\partial_{y_1}f)_1, (C_2f)_1\rangle + \langle (S_2f)_1, (C_2\partial_{y_1}f)_1\rangle\\
    & = \|(C_2f)_1\|^2 + 2 \langle (S_2f)_1,(C_2\partial_{y_1}f)_1\rangle +   \langle [(C_2)_1,(S_2)_1]f,\partial_{y_1}f\rangle .
\end{split}
\]
When calculating the commutator appearing on the right-hand side, we obtain multiples of $\si_3\partial_1 f$, $\si_4 \partial_2 f$, $\si_3 \partial_3 f$ and $\si_4 \partial_4 f$. All other terms cancel out. The commutator is thus controlled by $|M_2\grad_x f|$, and thus, using the Cauchy--Schwarz and Young inequalities, we find
\[
\|(S_2 f)_1 \|^2  \lesssim \|C_2 f\|^2 + \|C_2 \grad_y f\|^2 + \|\grad_y f\| \|M_2\grad_x f\|.
\]
Via a shift by $\pi/2$ in both $x_3$ and $x_4$, we  get the analogous estimate  also for the second component. Invoking then the pointwise estimate of Lemma \ref{L1} and applying  Young's inequality another time gives the first statement of the lemma.

The second statement follows by an application of Lemmas \ref{L1} and  \ref{L9} and an inspection of the dissipation functional.
\end{proof}

The first component of $S_2$ occurs also by commuting $S_1$ with the advection term. It appears thus also as an error term in the hypocoercivity method.

\begin{lemma}
    \label{L2-S}
    It holds that 
    \[
    [S_1,u\cdot\nabla_x]f =  \begin{pmatrix}
    -1 \\
    1
\end{pmatrix} (S_2f)_1 +Q_2f,
    \]
and the remainder $Q_2$ satisfies the estimate
\[
\beta_0 \|Q_2 f\|^2  \lesssim  \Psi(f).
\]
\end{lemma}

\begin{proof}
A direct computation shows that
\[
\begin{split}
    [S_1,u\cdot\nabla_x]f & = \begin{pmatrix}
        \si_3\co_1(\si_4\si_2\partial_1 f - \co_4\co_2\partial_3 f) \\
        \si_4\co_2(\si_3\si_1 \partial_2 f  -\co_3\co_1\partial_4f)
    \end{pmatrix} - \begin{pmatrix}
  \si_4\co_2(\si_3\si_1\partial_2 f - \co_3\co_1\partial_4f) \\
  \si_3\co_1(\si_4\si_2 \partial_1 f  -\co_4\co_2\partial_3f)
    \end{pmatrix} \\
    & \quad - \begin{pmatrix}
\co_3\si_1(\co_4\co_2\partial_1 f - \si_4\si_2\partial_3 f) \\
\co_4\si_2 (\co_3\co_1\partial_2 f - \si_3\si_1\partial_4 f)
\end{pmatrix} + \begin{pmatrix}
    \co_4\si_2(\co_3\co_1\partial_2 f - \si_3\si_1 \partial_4 f) \\
    \co_3\si_1 (\co_4\co_2\partial_1 f - \si_4\si_2\partial_3f)
\end{pmatrix}\\
& =     \begin{pmatrix}
    -1 \\
    1
\end{pmatrix} (S_2f)_1 +\begin{pmatrix}
    -\co_3\si_1 &\co_4\si_2 \\ \co_3\si_1 & -\co_4 \si_2
\end{pmatrix} C_1f,
\end{split}
\]
and the second claim of the lemma follows from the definition of the dissipation term.
\end{proof}

Later in the proof of the hypocoercivity estimate, it will be beneficial to understand the commutators between  commutators and their orthogonal complements.

\begin{lemma}
    \label{L21}
    For any $i,j\in\{1,2\}$, it holds that
    \[
    [(C_1)_i,(S_1)_i] = 0,
    \]
    and 
    \[
    |[(C_2)_i,(S_2)_j]f| \lesssim |M_2\grad_xf|.
    \]
\end{lemma}

\begin{proof}The first statement is readily checked. The derivation of the second one is tedious. Because there are no cancellations that we can exploit, we restrict ourself to considering  $(C_2)_1(S_2)_1$. We find
\begin{align*}
    (C_2)_1(S_2)_1f &=-  \co_1\co_2\si_3\si_4\left(\si_3\si_1(\si_4\si_2\partial_1f-\co_4\co_2\partial_3f) + \si_4\co_2(\si_3\co_1\partial_2f +\co_3\si_1\partial_4f)\right)\\
    & \quad + \co_1\si_2\si_3\co_4\left(\co_3\co_1(\si_4\si_2\partial_1 f - \co_4\co_2\partial_3f) -\si_4\co_2(\co_3\si_1\partial_2f +\si_3\co_1\partial_4f)\right)\\
    &\quad +\si_1\si_2\si_3\si_4\left(\si_3\co_1(\si_4\co_2\partial_1f +\co_4\si_2\partial_3 f) +\si_4\si_2(\si_3\si_1\partial_2f -\co_3\co_1\partial_4f)\right)\\
    &\quad -\co_1\si_2\co_3\si_4\left(\si_3\co_1(\co_4\si_2\partial_1f + \si_4\co_2\partial_3f) - \co_4\co_2(\si_3\si_1\partial_2f - \co_3\co_1\partial_4f)\right).
\end{align*}
Considering each of the terms individually, we notice that they are all controlled by $M_2\grad_x f$. The remaining commutators behave in the same way,   and thus, the statement follows.    
\end{proof}

We shall now discuss the third-order commutators.

\begin{lemma}
    \label{L3}
    It holds that
    \[
    [C_2,u\cdot\grad_x ]f = C_3 f+ R_3f,
    \]
    where 
\[
C_3 f = \begin{pmatrix}
 \co_3\co_1(\si_2^2-\co_2^2)\partial_1 f -2 \co_4\co_1\si_1\si_2 \partial_2 f\\
\co_4(\si_1^2-\co_1^2)\co_2 \partial_2 f  -2\co_3\co_2\si_1\si_2\partial_1 f
\end{pmatrix},
\]
and, if the coefficients  of  $\Psi$ satisfy the relations of Lemma \ref{L11}, the remainder $R_3$ satisfies the estimate
\[
\nu\alpha_2  \|R_3 f\|^2 \lesssim \Psi(f).
\]
\end{lemma}
\begin{proof}We focus on the first entry. The second one can be obtained by symmetry. We compute
\begin{align*}
    [(C_2)_1,u\cdot \grad_x]f & = \co_1\co_2\si_3\si_4 (\co_3\co_1\partial_2f-\si_3\si_1\partial_4f) -2\co_1\si_2\si_3\co_4(\si_3\si_1\partial_2f-\co_3\co_1\partial_4f)\\
    &\quad + \si_1 \si_2\si_3\si_4 (\co_4\co_2\partial_1 f - \si_4\si_2\partial_3 f) + \co_1\si_2 \co_3\si_4 (\si_4\si_2 \partial_1 f - \co_4\co_2\partial_3f)\\
    &\quad + \si_1\si_2\si_3\co_4(\si_4\co_2\partial_1 f +\co_4\si_2 \partial_3f) - \co_1\co_2\co_3\si_4(\si_4\co_2\partial_1f +\co_4\si_2\partial_3f)\\
    &\quad -\co_1^2 \si_3^2 (\co_4\co_2\partial_1 f - \si_4\si_2\partial_3f) + \si_1\co_1\si_3\co_3(\si_4\si_2\partial_1f - \co_4\co_2\partial_3f)\\
    &\quad -\si_1\co_2\co_3\si_4(\si_3\si_1\partial_2f -\co_3\co_1\partial_4f) - \si_2\co_2\si_4^2 (\co_3\si_1\partial_2 f+\si_3\co_1\partial_4f)\\
    &\quad -\si_2^2\si_4\co_4 (\si_3\co_1\partial_2 f +\co_3\si_1\partial_4f).
    \end{align*}
Most of the terms are immediately controlled by $M_1\grad_xf$ and $M_2\grad_xf$. We collect them in $R_3f$. The remaining terms can be simplified because by using the Pythagorean identity, we have
\begin{align}
\mel   \si_4^2\co_3\co_1(\si_2^2-\co_2^2)\partial_1 f -2 \si_3^2\co_4\co_1\si_1\si_2 \partial_2 f\\
&=  \co_3\co_1(\si_2^2-\co_2^2)\partial_1 f -2  \co_4\co_1\si_1\si_2 \partial_2 f  -\co_4^2\co_3\co_1(\si_2^2-\co_2^2)\partial_1 f +2 \co_3^2\co_4\co_1\si_1\si_2 \partial_2 f  .
\end{align}
The first two terms constitute $C_3$. The remaining two terms are again controlled by $M_1\grad_x f$, and can thus be added to $R_3f$. We then have
\[
|R_3f| \lesssim |M_1\grad_x f| + |M_2\grad_x f|.
\]
Invoking the estimates from Lemmas \ref{L9} and \ref{lemma:S2-bound}, we find the desired control on the remainder. 
\end{proof}

In the next lemma, we establish an unweighted control of the $\grad_x' = (\partial_1,\partial_2)^T$ gradient.

\begin{lemma}
    \label{L5}
    It holds that
\[
\int_{\T^2}|\grad_x' f |^2\, \dd y \lesssim \int_{\T^2}  \left(|M_1\grad_x  f|^2  +  |C_3 f|^2 + |\grad_x' \grad_y f|^2\right) \, \dd y.
\]
In particular, if the coefficients of $\Psi$ are as in Lemma \ref{L11}, then we have
\[
\beta_2 \|\grad_x'  f\|^2\lesssim \Psi(f).
\]
\end{lemma}

As stated in the lemma, we do not get a pointwise control on $\grad_x'f$. Instead, it is necessary to consider an integrated version. The reason for this is that the derivatives $\partial_1f$ and $\partial_2f$ in $C_3f$ have both degenerating prefactors. We eliminate these with the help of the following Hardy--Poincar\'e inequality.

\begin{lemma}
    \label{L20}
Let $\omega:\T^2\to\R$ denote a smooth and bounded function that degenerates linearly at a finite number of points $Y_1,\dots,Y_N\in \T^2$ and is bounded away from zero elsewhere, 
\[
|\omega(y)|\sim \begin{cases} |y-Y_i|& \mbox{for $y$ near $Y_i$},\\ 1 &\mbox{elsewhere.}\end{cases}
\]

Then 
\begin{align*}
\int_{\T^2} g^2 \dd y& \lesssim \int_{\T^2} \omega^2  g^2 \dd y  + \int_{\T^2} |\grad_y g|^2 \dd y.
\end{align*}
\end{lemma}

\begin{proof}We denote by $\rho$ the minimal (periodic) distance between the points of degeneracy,
\[
\rho = \min_{i\not=j}|Y_i-Y_j|\mod 2\pi.
\]
Then, because $\omega$ is not degenerating away from the distinguished points, 
\begin{equation}
    \label{25}
\int_{\T^2\setminus \cup_{i=1}^N B_{\frac{\rho}4}(Y_i)} g^2\,\dd y \lesssim \int_{\T^2}\omega^2g^2\, \dd y.
\end{equation}

In order to obtain the estimate near the points of degeneracy, is is enough to focus on one point, say $Y_1$. By translation, we may shift this point to the origin, that is, $Y_1=0$. We consider a smooth cut-off function $\eta=\eta(y)\in[0,1]$ such that
    \[
    \eta(y)=\begin{cases}
        1&\quad\mbox{for }|y|\le \frac{\rho}4,\\
        0&\quad\mbox{for }|y|\ge\frac{\rho}2.
    \end{cases}
    \]
Then, via an integration by parts, we obtain
\begin{align*}
    \int_{B_{\frac{\rho}4}(0)} g^2\, \dd y \le \int_{\T^2}\eta^2 g^2\, \dd y = \frac12 \int_{\T^2}\eta^2 (\grad\cdot y)   g^2\,\dd y = -  \int_{\T^2}\eta \grad\eta\cdot y g^2\, \dd y -\int_{\T^2} \eta^2 y \cdot g\grad g\, \dd y.
\end{align*}
Because $|\grad\eta|\lesssim 1$, we may use Young's inequality and the bound $|y|\lesssim |\omega(y)|$ to deduce that
\begin{equation}
    \label{26}
        \int_{B_{\frac{\rho}4}(0)} g^2\, \dd y \le \int_{\T^2}\eta^2 g^2\, \dd y \lesssim \int_{\T^2} \omega^2g^2 \, \dd y +\int_{\T^2} |\grad_y g|^2\, \dd y.
\end{equation}

We eventually combine the bound \eqref{26} near the points of degeneracy and the bound \eqref{25} away from the points of degeneracy to deduce the desired global estimate.
\end{proof}

\begin{proof}[Proof of Lemma \ref{L5}]
By symmetry, it is enough to focus on the first entry. Thanks to the Pythagorean identity and the definition of $M_2$, it is furthermore sufficient to control $\co_3 \partial_1 f$. This can be done with the help of Lemma \ref{L20}, once the commutator is sufficiently well understood. We expand
\begin{align*}
    |C_3 f|^2 & =  \left(\co_1^2 (\si_2^2-\co_2^2)^2 +  4\si_1^2 \co_2^2\si_2^2 \right) \co_3^2 (\partial_1f)^2  + \left(\co_2^2(\si_1^2-\co_1^2)^2 +4\si_2^2 \co_1^2\si_1^2\right) \co_4^2 (\partial_2 f)^2 \\
    &\quad -4 \si_1\si_2\left(\co_1^2(\si_2^2-\co_2^2) + \co_2^2(\si_1^2-\co_1^2)\right)\co_3\co_4\partial_1f\partial_2f.
\end{align*}
The mixed terms in the second line are all controlled by $M_1\grad_xf$, so that
\[
\left(\co_1^2 (\si_2^2-\co_2^2)^2 +  \si_1^2 \co_2^2\si_2^2 \right) \co_3^2 (\partial_1f)^2 \lesssim |C_3f|^2 + |M_1\grad_xf|^2.
\]
In order to eliminate the degenerating prefactor of $\co_3\partial_1 f$, we apply the Hardy--Poincar\'e inequality from Lemma \ref{L20}, noticing that the degeneracy occurs linearly at the points
\[
(x_1+y_2,x_2+y_1)\in\left\{ \left(\pm \frac{\pi}2,0\right),\left(\pm\frac{\pi}2,\pm\frac{\pi}2\right),\left(0,\pm\frac{\pi}4\right),\left(0,\pm\frac{3\pi}4\right)\right\}.
\]
This proves the lemma.
\end{proof}

Analogously to the computation of the third-order commutator $C_3$, we have to consider  the commutator of the second orthogonal complement.

\begin{lemma}
\label{L3-S}
Let the coefficients of $\Psi$ be as in Lemma \ref{L11}. Then 
\[
\beta_2 \|[S_2,u\cdot \grad_x]f\|^2 \lesssim \Psi(f)
\]
 \end{lemma}

\begin{proof}
The computation is quite tedious, and thus, we only want to give an idea on how to derive the result. More specifically, we will restrict our attention to the first component of $S_2$, and because there are no cancellations that we have to exploit, we will focus on the part $(S_2)_1(u\cdot \grad_x)$ of the commutator. For this part, we compute
\begin{align*}
    (S_2)_1(u\cdot \grad_x )f & = \co_1\si_2\si_3\si_4 (\co_3\co_1\partial_2f - \si_3\si_1\partial_4f) + \co_1\co_2\si_3\co_4(\si_3\si_1\partial_2f - \co_3\co_1 \partial_4f)\\
    &\quad -\si_1\co_2\si_3\si_4 (\co_4\co_2\partial_1 f - \si_4\si_2\partial_3f)  - \co_1\co_2\co_3\si_4(\si_4\si_2\partial_1f - \co_4\co_2\partial_3f).
\end{align*}
The terms involving derivatives with respect to $x_1$ and $x_2$ are trivially bounded by $\grad_x'f$. The respective first terms with derivatives $x_3$ and $x_4$ are controlled by $M_2\grad_x f$, while the remaining two terms are controlled by $M_1\grad_xf$. The argumentation for all other terms in the full commutator is identical. It thus holds that
\[
|[S_2,u\cdot \grad_x ]f| \lesssim |M_1\grad_x f|  + |M_2\grad_x f| + |\grad_x' f|,
\]
and thus, the stated norm estimate is a consequence of Lemmas \ref{L9}, \ref{lemma:S2-bound}, and \ref{L5}.
\end{proof}

We finally need to investigate the commutators of fourth order.

\begin{lemma}\label{lemma:C4}
We have the estimates   \[
   |[C_3,u\cdot\nabla_x]f|+     |[\grad_x' ,u\cdot \grad_x ]f| \lesssim |M_2\grad_x f| + |\grad_x' f|.
    \]
In particular, if the coefficients of $\Psi$ are solving the relations stated in Lemma \ref{L11} these commutators satisfy the estimate
    \[
\beta_2  \|[C_3,u\cdot\nabla_x]f\|^2  + \beta_2 \|[\grad_x' ,u\cdot \grad_x ]f\|^2\lesssim \Psi(f).
\]
\end{lemma}

\begin{proof}By symmetry, it is enough to focus on the first component of the commutators. 
Direct computations yield that
\begin{align*}
    [(C_3)_1,u\cdot \grad_x ]f & = \co_1(\si_2^2-\co_2^2)\co_3(\co_3\co_1\partial_2f - \si_3\si_1\partial_4 f) - 2\co_1\si_1\si_2\co_4(\co_4\co_2\partial_1f-\si_4\si_2\partial_3f)\\
   & \quad + \si_1\si_2(\si_2^2-\co_2^2)\co_3\co_4\partial_1 f -4 \si_1\co_1\si_2\co_2\co_3^2 \partial_1f + \co_1\co_2(\si_2^2-\co_2^2)\si_3 \si_4 \partial_1 f\\
   &\quad -2(\si_1^2-\co_1^2)\si_2^2\co_4^2\partial_2 f +2 \si_1^2\co_1\co_2\co_3\co_4\partial_2 f - 2\si_1\co_1^2\si_2\si_3\si_4 \partial_2f,
\end{align*}
and 
\[
[\partial_1,u\cdot \grad_x]f = \partial_1u\cdot \grad_x f = \co_3\co_1\partial_2f - \si_3\si_1\partial_4f.
\]
All the terms that involve derivatives with respect to $x_1$ and $x_2$ are controlled by $|\grad_x'f|$. The remaining terms are bounded by $M_1\grad_xf$, so that,
\[
| [C_3 ,u\cdot \grad_x ]f| + |[\grad_x',u\cdot \grad_x]f| \lesssim |M_1\grad_xf| + |\grad_x' f|.
\]
The stated estimate then follows by invoking Lemmas \ref{L9} and \ref{L5}, and the ordering of the coefficients in Lemma \ref{L11}.
\end{proof}

\subsection{Decay of the Lyapunov functional} \label{SS2}

In this section we prove Proposition \ref{P1} and give details on how to control all the error terms arising from the time derivative of the energy functional \eqref{energy}, whose definition we include here again for the convenience of the reader,
\[
\begin{split}
    \Phi(f) & = \frac{1}{2}\|f\|^2 + \frac{\alpha_0}{2}\|\nabla_yf\|^2 +\frac{\kappa\delta}{2}\|\nabla_xf\|^2 + \beta_0\langle \nabla_yf,C_1f\rangle + \frac{\alpha_1}{2}\|C_1f\|^2 + \frac{\gamma_1}{2}\|S_1f\|^2 \\
    & \quad +\beta_1\langle C_1f,C_2f\rangle +  \frac{\alpha_2}{2}\|C_2f\|^2 + \frac{\gamma_2}2 \|S_2f\|^2 + \beta_2 \langle C_2f,C_3f\rangle + \frac{\alpha_3}2\|\nabla_x'f\|^2.
\end{split}
\]
We address the derivative addend by addend and express all the errors in terms of elements of the dissipation functional \eqref{diss}, that again, for ease of reference, we recall here:
\[
\begin{split}
    \Psi(f)  & = \nu \|\grad_y f\|^2 + \nu\alpha_0 \|\laplace_y f\|^2 + \beta_0 \|C_1f\|^2 + \nu \alpha_1 \|C_1\grad_y f\|^2 + \nu \gamma_1 \|S_1 \grad_y f\|^2 \\
    & \quad + \beta_1 \|C_2f\|^2 +\nu \alpha_2 \|C_2 \grad_y f\|^2 + \nu \gamma_2 \|S_2\grad_y f\|^2 + \beta_2 \|C_3 f\|^2 + \nu \alpha_3 \|\nabla_x'\grad_y f\|^2 \\
    & \quad + \kappa \left[ \|\nabla_xf\|^2 + \alpha_0\|\nabla_{x}\grad_yf\|^2 + \kappa\delta\|\laplace_x f\|^2 + \alpha_1\|C_1\grad_x f\|^2 + \gamma_1\|S_1\grad_x f\|^2 \right. \\
    & \qquad \left. + \alpha_2\|C_2\grad_x f\|^2 + \gamma_2\|S_2\grad_x f\|^2 + \alpha_3\|\grad_x'\grad_x f\|^2 \right].
\end{split}
\]

\begin{proof}[Proof of Proposition \ref{P1}] We address the rate of change of the energy functional term by term.

\medskip
    
\noindent
\emph{The $1$-term}. The first addend in \eqref{energy} corresponds to the $L^2$ norm of the solution to the advection-diffusion equation \eqref{18}, so we use the structure of the equation to compute
\[
\frac{1}{2}\ddt \|f\|^2 = \langle \partial_tf,f\rangle = -\langle u\cdot\nabla_xf,f\rangle +  \nu\langle\Delta_yf,f\rangle + \kappa \langle \laplace_xf,f\rangle.
\]
Now integrating by parts in the three addends and using that $u$ is divergence free, namely $\langle u\cdot\nabla_x f, f\rangle =- \langle f, u\cdot\nabla_x f\rangle =0$, we get the energy identity
\begin{equation}
    \label{9}
\ddt\frac{1}{2} \|f\|^2 +\nu \|\nabla_y f\|^2 + \kappa\|\nabla_xf\|^2 = 0.
\end{equation}

\medskip
\noindent
\emph{The $\alpha_0$-term}. The next term in \eqref{energy} corresponds to the $L^2$ norm of the gradient in $y$. Following an analogous procedure to the previous case, we find that
\[
\frac{1}{2}\ddt \|\nabla_y f\|^2 + \nu\|\Delta_y f\|^2 + \kappa\|\nabla_x\grad_yf\|^2 = -\langle \nabla_yf,\nabla_y(u\cdot\grad_xf)\rangle = -\langle \nabla_y f, C_1f\rangle,
\]
where the missing term in the commutator $C_1 = [\grad_y,u\cdot\grad_x]$, see Lemma \ref{L15}, could be introduced without error thanks to the skew-symmetry of $u\cdot \grad_x$. Now notice that not all terms can directly be part of the dissipation terms in \eqref{diss}.  Instead there is an error term in the right-hand side. In order to control this error term by the  energy dissipation functional $\Psi(f)$ we use Hölder's and Young's inequalities as follows,
\[
\alpha_0|\langle \nabla_y f, C_1f\rangle| \leq  \frac{\alpha_0}{\sqrt{\nu\beta_0}}\sqrt{\nu}\|\nabla_yf\|\sqrt{\beta_0}\|C_1f\| \leq \frac{\alpha_0}{\sqrt{\nu\beta_0}}\left(\nu \|\nabla_y f\|^2 + \beta_0\|C_1f\|^2\right).
\]
Therefore, using the definition of $\Psi(f)$ in \eqref{diss}, we obtain the estimate
\begin{equation}
\label{10}
\ddt \frac{\alpha_0 }2  \|\grad_y f\|^2 + \nu \alpha_0 \|\laplace_y f\|^2 + \kappa\alpha_0\|\nabla_{x,y}^2f\|^2 \le \frac{\alpha_0}{\sqrt{\nu \beta_0}}\Psi(f).
\end{equation}

\medskip
\noindent
\emph{The $\kappa\delta$-term.} This term can be treated very similarly to the $\alpha_0$-term. Via some integrations by parts, we obtain
\[
\ddt \frac12 \|\grad_x f\|^2 +\nu \|\grad_x\grad_y f\|^2 +\kappa \|\laplace_x f\|^2  = -\la\grad_x f, [\grad_x,u\cdot\grad_x]f\ra .
\]
The first dissipative term on the left-hand side can be dropped, because there is already a better term in $\Psi$ coming from the dissipation of the $\alpha_0$-term. The commutator that occurs on the right-hand side can be roughly bounded by the $x$-gradient, $|[\grad_x,u\cdot\grad_x]f| = |(\grad_xu)^T\grad_xf| \lesssim |\grad_x f|$, and thus, with regard to the definition of the dissipation functional in \eqref{diss}, we find
\[
\ddt\frac{\kappa\delta}2 \|\grad_x f\|^2  + \kappa^2 \delta \|\laplace_x f\|^2 \lesssim \delta\Psi(f).
\]

\medskip
\noindent
\emph{The $\beta_0$-term}. Next in order in \eqref{energy}   is the $\beta_0$-term . This is the non-symmetric term in the energy functional that we come across, and its estimate is a bit more involved. We start by using the structure of equation \eqref{18},
\[
\begin{split}
    \ddt \langle \nabla_yf,C_1f \rangle & = -\langle \nabla_y(u\cdot\nabla_xf),C_1f \rangle - \langle \nabla_yf,C_1(u\cdot\nabla_xf) \rangle + \nu \langle \nabla_y\Delta_yf,C_1f \rangle  \\
    & \quad + \nu\langle \nabla_yf,C_1\Delta_yf \rangle + \kappa \langle \nabla_y\laplace_x f,C_1f \rangle + \kappa \langle \nabla_yf,C_1\laplace_x f \rangle.
\end{split}
\]

For the terms concerning the transport $u\cdot\nabla_x$ we can write
\[
\begin{split}
    -\langle \nabla_y(u\cdot\nabla_xf),C_1f \rangle - \langle \nabla_yf,C_1(u\cdot\nabla_xf) \rangle & = -\|C_1f\|^2 - \langle (u\cdot\nabla_x)\nabla_yf,C_1f \rangle \\
    & \quad - \langle \nabla_yf,[C_1,u\cdot\nabla_x]f \rangle - \langle \nabla_yf,(u\cdot\nabla_x)C_1f \rangle \\
    & = -\|C_1f\|^2 -  \langle \nabla_yf,[C_1,u\cdot\nabla_x]f \rangle,
\end{split}
\]
where we used the definition of the first commutator $C_1f = [\nabla_y,u\cdot\nabla_x]f  $ in Lemma \ref{L15}. Let us address the  terms coming from the dissipation in $y$. Integrating by parts, we find
\[
    \nu \langle \nabla_y\Delta_yf,C_1f \rangle + \nu\langle \nabla_yf,C_1\Delta_yf \rangle = - \nu \langle \Delta_yf,(\nabla_y\cdot C_1)f \rangle - 2\nu\langle C_1\cdot\nabla_yf,\Delta_yf \rangle.
\]
Using the observation $\nabla_y\cdot C_1= -u\cdot\nabla_x$ and the skew-symmetry  of that operator, we can rewrite the first term on the right-hand side as
\[
-\nu\langle \Delta_yf,(\nabla_y\cdot C_1)f \rangle = -  \nu\langle \grad_yf,C_1f \rangle.
\]
Arguing similarly for the $\kappa$-terms, we find
\[
\kappa \la \grad_y\laplace_x f,C_1f\ra + \kappa \la \grad_y f,C_1\laplace_x f\ra  = -\kappa \la \grad_x f,(\grad_xu)^T \grad_x f\ra  -2\kappa \la \laplace_x f, C_1\cdot \grad_yf\ra,
\]
where we yet again used the relation $\nabla_y\cdot C_1= -u\cdot\nabla_x$.
Therefore, putting everything together, we obtain
\[
\begin{split}
    \frac{\dd}{\dd t}\langle \nabla_yf,C_1f \rangle + \|C_1f\|^2 & = -\nu\langle \grad_yf,C_1f \rangle - 2\nu\langle \Delta_yf, C_1\cdot\nabla_yf \rangle  \\
    & \quad - \langle \nabla_yf, [C_1,u\cdot\nabla_x]f \rangle -\kappa \la \grad_x f,(\grad_xu)^T \grad_x f\ra  \\
    & \quad   -2\kappa \la \laplace_x f, C_1\cdot \grad_yf\ra.
\end{split}
\]
We control the error terms in the right hand side of the energy identity using Hölder's and Young's inequality as before. For instance, for the $\nu$-terms, we have
\begin{align*}
 \nu\beta_0|\langle \grad_y f,C_1f \rangle| &\leq \sqrt{\nu\beta_0} \left(\nu \|\grad_y f\|^2 + \beta_0\|C_1 f\|^2\right),\\
\nu\beta_0|\langle \Delta_yf,C_1\cdot\nabla_yf \rangle| &\leq \frac{\beta_0}{\sqrt{\alpha_0\alpha_1}}\left(\nu\alpha_0\|\Delta_yf\|^2 + \nu\alpha_1 \|C_1\nabla_yf\|^2\right).
\end{align*}
Regarding the advection term, we use  the decomposition of the second commutator from Lemma \ref{L2} and find
\[
\begin{split}
    \beta_0|\langle \nabla_yf, [C_1,u\cdot\nabla_x]f \rangle|  & \leq \beta_0 \|\grad_y f\|\|R_2 f\| + \beta_0 \|\grad_y f\|\|C_2 f\|\\
    &\le \frac{\beta_0}{\nu\sqrt{\alpha_1}}\left(\nu\|\grad_y f\|^2 + \nu\alpha_1 \|R_2 f\|^2\right) + \frac{\beta_0}{\sqrt{\nu\beta_1}}\left(\nu \|\grad_y f\|^2 +\beta_1 \|C_2 f\|^2\right).
\end{split}
\]
Finally, for the $\kappa$-terms we find analogously
\begin{gather*}
    \kappa\beta_0| \la \grad_x f,(\grad_xu)^T \grad_x f\ra |  \lesssim \kappa\beta_0  \|\grad_x f\|^2 ,\\
\kappa \beta_0 | \la \laplace_x f, C_1\cdot \grad_yf\ra|  \le \frac{\beta_0}{\sqrt{\nu\alpha_1}}\frac{1}{\sqrt{\delta}} \left(\kappa^2 \delta \|\laplace_x f\|^2 + \nu\alpha_1 \|C_1\grad_y f\|^2\right).
\end{gather*}
Therefore, using the estimate of the remainder from Lemma \ref{L2} and the definition of the dissipation functional \eqref{diss}, we have that
\begin{equation}
    \label{11}
    \ddt \beta_0 \la \grad_y f ,C_1f\ra +\beta_0 \|C_1 f\|^2 \lesssim \left(\sqrt{\nu\beta_0} + \frac{\beta_0}{\sqrt{\alpha_0\alpha_1}} + \frac{\beta_0}{\nu\sqrt{\alpha_1}} + \frac{\beta_0}{\sqrt{\nu\beta_1}} +\beta_0 +\frac{\beta_0}{\nu\sqrt{\alpha_1}}\frac{1}{\sqrt{\delta}} \right)\Psi(f).
\end{equation}

\medskip

\noindent   
\emph{The $\alpha_1$-term}. The following term in the energy functional \eqref{energy} to study is $\alpha_1$. Proceeding as before, we find the following energy identity
\[
\begin{split}
\frac{1}{2}\frac{\dd}{\dd t}\|C_1f\|^2 + \nu\|C_1\nabla_yf\|^2 + \kappa\|C_1\nabla_xf\|^2 & = -2\nu\langle \nabla_yC_1\nabla_yf,C_1f\rangle -\langle C_1f,[C_1,u\cdot\nabla_x]f\rangle  \\
& \quad -\kappa\la \grad_xC_1\grad_x f,C_1f\ra - \kappa\la C_1\grad_x f,\grad_x C_1,f\ra ,
\end{split}
\]
where  we used the fact that the commutator   $C_1$ commutes with its derivatives  $\partial_{y_i}C_1$   for any $i\in\{1,2\}$, see Lemma \ref{L21},  so that $\la C_1 \partial_{y_i} f,\partial_{y_i}C_1 f\ra = \la \partial_{y_i}C_1\partial_{y_i}f,C_1f\ra$.

We discuss the estimates for the right-hand side term by term. First, we notice that
\begin{align*}
 \nu\alpha_1|\langle \nabla_yC_1\nabla_yf,C_1f\rangle| &\leq \frac{\alpha_1\sqrt{\nu}}{\sqrt{\beta_0\gamma_1}}\left(\beta_0\|C_1f\|^2 + \nu\gamma_1\|S_1\nabla_yf\|^2\right), 
\end{align*}
where the first orthogonal complement $S_1$ was introduced in \eqref{34}. Invoking   the decomposition of the second commutator from Lemma \ref{L2}, we furthermore have
\[
\begin{split}
    \alpha_1|\langle C_1f,[C_1,u\cdot\nabla_x]f\rangle| & \leq \alpha_1\|C_1f\|\|C_2f\| + \alpha_1\|C_1f\|\|R_2f\| \\
    & \leq \frac{\alpha_1}{\sqrt{\beta_0\beta_1}}\left(\beta_0\|C_1f\|^2 + \beta_1\|C_2f\|^2 \right) + \frac{\sqrt{\alpha_1}}{\sqrt{\nu\beta_0}}\left(\beta_0\|C_1f\|^2 + \nu\alpha_1\|R_2f\|^2 \right).
\end{split}
\]
We finally address the $\kappa$-terms,
\begin{gather*}
    \kappa \alpha_1 |\la \grad_xC_1\grad_x f,C_1f\ra |  \lesssim \kappa \alpha_1 \|\laplace_x f\|\|C_1f\| \le \frac{ \alpha_1}{\sqrt{\beta_0}}\frac{1}{\sqrt{\delta}}\left(\kappa^2 \delta \|\laplace_x f\|^2 +\beta_0 \|C_1f\|^2\right),\\
    \kappa \alpha_1 |\la C_1\grad_x f,\grad_x C_1 f\ra| \lesssim \kappa\alpha_1 \|C_1\grad_xf\|\|\grad_xf\| \le \sqrt{\alpha_1}\left(\kappa \alpha_1 \|C_1\grad_x f\|^2 + \kappa \|\grad_x f\|^2\right).
\end{gather*}

Putting things together and using the estimate on the remainder term of Lemma \ref{L2}, we obtain
\begin{equation}
    \label{12} 
    \begin{aligned}\mel
        \ddt \frac{\alpha_1}2 \|C_1 f\|^2 + \nu \alpha_1 \|C_1\grad_y f\|^2   + \kappa \alpha_1 \|C_1\grad_x f\|^2 \\
        & \lesssim \left(\frac{\alpha_1 \sqrt{\nu}}{\sqrt{\beta_0\gamma_1}} + \frac{\sqrt{\alpha_1}}{\sqrt{\nu\beta_0}}+ \frac{\alpha_1}{\sqrt{\beta_0\beta_1}} +\frac{ \alpha_1}{\sqrt{\beta_0}}\frac{1}{\sqrt{\delta}} + \sqrt{\alpha_1} \right)\Psi(f).
    \end{aligned}
\end{equation}

\medskip

\noindent
\emph{The $\gamma_1$-term}. The estimate for the $\gamma_1$-term in the energy functional \eqref{energy} is completely analogous. We find that 
\[
\begin{split}
\frac{1}{2}\frac{\dd}{\dd t}\|S_1f\|^2 + \nu\|S_1\nabla_yf\|^2 + \kappa\| S_1\nabla_xf\|^2 & = -2\nu\langle \nabla_yS_1\nabla_yf,S_1f\rangle -\langle S_1f,[S_1,u\cdot\nabla_x]f\rangle  \\
& \quad -\kappa\la \grad_xS_1\grad_x f,S_1f\ra - \kappa \la S_1\grad_x f,\grad_xS_1f\ra. 
\end{split}
\]
In spirit, the estimates for the terms on the right-hand side are identical to those obtained for the $\alpha_1$-terms. This time, we only have to invoke Lemma \ref{L2-S} for the commutator of the orthogonal complement $S_1$. We have
\begin{align*}
    \nu\gamma_1|\langle \nabla_yS_1\nabla_yf,S_1f\rangle| &\leq \frac{\gamma_1}{\alpha_1} \left(\nu\alpha_1\|S_1f\|^2 + \nu\alpha_1\|C_1\nabla_yf\|^2\right),\\
   \gamma_1|\langle S_1f,[S_1,u\cdot\nabla_x]f\rangle| &\le\frac{\gamma_1}{\nu\sqrt{\alpha_1\alpha_2}} \left(\nu\alpha_1\|S_1f\|^2 + \nu\alpha_2\|S_2f\|^2\right)\\
   &\quad + \frac{\gamma_1}{\sqrt{\nu\alpha_1\beta_0}} \left(\nu\alpha_1\|S_1f\|^2 + \beta_0\|Q_2f\|^2\right),\\
 \kappa\gamma_1  | \la \grad_xS_1\grad_x f,S_1f\ra| &\lesssim \frac{\gamma_1}{\sqrt{\nu\alpha_1}}\frac{1}{\sqrt{\delta}}\left(\kappa^2\delta\|\laplace_x f\|^2 + \nu\alpha_1\|S_1f \|^2\right),\\
 \kappa\gamma_1   | \la S_1\grad_x f,\grad_xS_1f\ra|&\lesssim \sqrt{\gamma_1}\left(\kappa\gamma_1\|S_1\grad_x f\|^2 +\kappa \|\grad_x f\|^2\right).
\end{align*}
Using the estimates from Lemmas \ref{L9}, \ref{lemma:S2-bound} and \ref{L2-S}, we eventually arrive at
\begin{equation}
    \label{13}
    \begin{aligned}\mel
    \ddt \frac{\gamma_1}2\|S_1 f\|^2 + \nu\gamma_1 \|S_1\grad_y f\|^2   + \kappa\gamma_1 \|S_1\grad_x f\|^2 \\
    & \lesssim \left(\frac{\gamma_1}{\alpha_1}+\frac{\gamma_1}{\sqrt{\nu\alpha_1\beta_0}} +\frac{\gamma_1}{\nu\sqrt{\alpha_1\alpha_2}}  +\frac{\gamma_1}{\sqrt{\nu\alpha_1}}\frac{1}{\sqrt{\delta}} + \sqrt{\gamma_1}  \right) \Psi(f).
    \end{aligned}
\end{equation}

\medskip
\noindent
\emph{The $\beta_1$-term}. As before, the non-symmetric terms require a bit of a more involved analysis, so let us present the result with more detail. We can split the time derivative in the following way
\[
\begin{split}
   \frac{\dd}{\dd t}\langle C_1f,C_2f\rangle & = -\la C_1(u\cdot\nabla_xf),C_2f\ra - \la C_1f,C_2(u\cdot\nabla_xf)\ra \\
   & \quad + \nu \la C_1\Delta_yf, C_2f\ra + \nu \la C_1f, C_2\Delta_yf\ra \\
   & \quad + \kappa \la  C_1\laplace_x f,C_2f\ra + \kappa\la  C_1f,C_2\laplace_x f \ra.
\end{split}
\]
For the transport terms, we repeat essentially our argument from the $\beta_0$-term. We can write with the help of the Commutator Lemmas \ref{L2} and \ref{L3},
\begin{align*}
    \mel     -\la C_1(u\cdot\nabla_xf),C_2f\ra - \la C_1f,C_2(u\cdot\nabla_xf)\ra \\
    & = - \la [C_1,u\cdot\nabla_x]f,C_2f \ra - \la u\cdot\nabla_x (C_1f),C_2f\ra   -\la C_1f,[C_2,u\cdot\nabla_x]f \ra - \la C_1f,u\cdot\nabla_x(C_2f) \ra \\
    & = -\|C_2f\|^2 - \la R_2f,C_2f\ra - \la C_1f,C_3f\ra - \la C_1f,R_3f\ra .
\end{align*}
For the $y$-dissipation terms,  we apply multiple integrations by parts and the   identities $\Delta_y C_1 = -C_1$ and $\Delta_y C_2 = -C_2$ to obtain
\begin{align*}
 \mel  \nu \la C_1\Delta_yf, C_2f\ra + \nu \la C_1f, C_2\Delta_yf\ra\\
 & = -\nu\la \nabla_yC_1\nabla_yf,C_2f \ra - \nu\la C_1\nabla_yf,\nabla_yC_2f \ra   -2\nu\la C_1\nabla_yf,C_2\nabla_yf \ra\\
 &\quad - \nu \la C_1f,\nabla_yC_2\nabla_yf \ra  - \nu\la \nabla_yC_1f,C_2\nabla_yf \ra \\
    & =  2\nu\la \nabla_yC_1f,\nabla_yC_2f \ra +  \nu\la C_1f,\Delta_yC_2f \ra    -2 \nu\la C_1\nabla_yf,C_2\nabla_yf \ra + \nu\la\Delta_yC_1f,C_2f\ra \\
    & =  2\nu\la \nabla_yC_1f,\nabla_yC_2f \ra  -2  \nu\la C_1f,C_2f \ra     -2 \nu\la C_1\nabla_yf,C_2\nabla_yf \ra.
\end{align*}
 Finally, for the dissipation terms coming from the $x$-Laplacian, we invoke the identical argumentation to get
\begin{align*}
 \mel       \kappa \la  C_1\laplace_x f,C_2f\ra + \kappa\la  C_1f,C_2\laplace_x f \ra  \\
 & = 2\kappa \la \grad_x C_1 f,\grad_x C_2f\ra -\kappa \la C_1f,C_2f\ra -2\kappa \la C_1\grad_x f,C_2\grad_xf\ra.
\end{align*}
Therefore,  putting everything together, we find
\[
\begin{split}
\frac{\dd}{\dd t}\langle C_1f,C_2f\rangle + \|C_2f\|^2 & =   2\nu\langle C_1\nabla_yf,C_2\nabla_yf \rangle - 2\nu\la C_1f,C_2f \ra - 2\nu\langle \nabla_yC_1f,\nabla_yC_2f \rangle \\
& \quad  -2\kappa \la \grad_x C_1 f,\grad_x C_2f\ra -\kappa \la C_1f,C_2f\ra -2\kappa \la C_1\grad_x f,C_2\grad_xf\ra \\
       &\quad  -\langle R_2f,C_2f\rangle  -  \la C_1f,C_3f\ra - \la C_1f,R_3f\ra.
\end{split}
\]

To estimate the terms in the right-hand side we proceed with a combination of Hölder's and Young's inequality as before. First, for the $\nu$-terms, we obtain
\begin{gather*}
\nu\beta_1|\langle C_1\nabla_yf,C_2\nabla_yf \rangle| \leq \frac{\beta_1}{\sqrt{\alpha_1\alpha_2}} \left(\nu\alpha_1\|C_1\nabla_yf\|^2 + \nu\alpha_2\|C_2\nabla_yf\|^2\right),\\
\nu\beta_1|\langle C_1f,C_2f \rangle| \leq \frac{\nu\sqrt{\beta_1}}{\sqrt{\beta_0}} \left(\beta_0\|C_1f\|^2 + \beta_1\|C_2f\|^2\right),\\
    \nu\beta_1|\langle \nabla_yC_1f,\nabla_yC_2f \rangle| \leq \frac{
\beta_1}{\sqrt{\alpha_1\alpha_2}}\left(\nu\alpha_1\|S_1f\|^2 + \nu\alpha_2\|S_2f\|^2 \right),
\end{gather*}
Our estimates for the advection terms are inspired by Lemmas \ref{L2} and \ref{L3},
\begin{gather*}
\beta_1|\langle R_2f,C_2f\rangle| \leq \frac{\sqrt{\beta_1}}{\sqrt{\nu\alpha_1}}\left(\beta_1\|C_2f\|^2 + \nu\alpha_1\|R_2f\|^2\right),\\
\beta_1|\langle C_1f,C_3f\rangle| \leq \frac{\beta_1}{\sqrt{\beta_0\beta_2}}\left(\beta_0\|C_1f\|^2 + \beta_2\|C_3f\|^2\right),\\
\beta_1|\langle C_1f,R_3f\rangle| \le\frac{\beta_1}{\sqrt{\nu\alpha_2\beta_0}}\left(\beta_0\|C_1f\|^2 + \nu\alpha_2\|R_3f\|^2\right).
\end{gather*}
Finally, for the $\kappa$-terms we have  
\begin{gather*}
   \kappa \beta_1 |\la \grad_x C_1 f,\grad_x C_2f\ra |\lesssim \kappa\beta_1  \|\grad_x f\|^2 ,\\
   \kappa \beta_1 |\la C_1f,C_2f\ra|\lesssim  \kappa\beta_1 \|\grad_x f\|^2 ,\\
   \kappa \beta_1 |\la C_1\grad_x f,C_2\grad_xf\ra|\lesssim \frac{\beta_1}{\sqrt{\alpha_1\alpha_2}}\left(\kappa \alpha_1\|C_1\grad_x f\|^2 + \kappa \alpha_2 \|C_2\grad_xf\|^2\right).
\end{gather*}
 All in all we can conclude by combining everything and applying Lemmas \ref{L9}, \ref{L2},  \ref{lemma:S2-bound} and \ref{L3},
\begin{equation}
    \label{14a}
    \begin{aligned}
    \mel  \ddt\beta_1 \la C_1f,C_2f\ra + \beta_1 \|C_2 f\|^2 \\
    & \lesssim \left(\frac{\beta_1}{\sqrt{\alpha_1\alpha_2}} +\frac{\nu \sqrt{\beta_1} }{\sqrt{\beta_0}}    +\frac{\sqrt{\beta_1}}{\sqrt{\nu\alpha_1}} + \frac{\beta_1}{\sqrt{\beta_0\beta_2}} + \frac{\beta_1}{\sqrt{\nu\alpha_2\beta_0}} + \beta_1  \right) \Psi(f).
    \end{aligned}
\end{equation}

\medskip

\noindent
\emph{The $\alpha_2$-term}. We proceed with an analogous computation as for the $\alpha_1$-terms. We only have to take into account that the second-order commutators are not commuting, cf.~Lemma \ref{L21}. We find the identity
\begin{align*}
 \mel    \ddt \frac12 \|C_2f\|^2 + \nu \|C_2\grad_y f\|^2 +\kappa \|C_2\grad_x f\|^2\\
 &= -2\nu \la \grad_yC_2 \grad_yf,C_2f\ra +\nu \la [\grad_y C_2,C_2] f,\grad_y f\ra -2\kappa \la C_2\grad_x f,\grad_x C_2f\ra\\
 &\quad +\kappa \la \grad_x f,[\grad_x C_2,C_2]f\ra -\la[C_2,u\cdot \grad_x ]f,C_2f\ra.
\end{align*}

To bound the  terms on the right-hand side, we notice that (modulo signs), the entries of $\grad_y C_2$ are given by those of $S_2$, and therefore,
\[
\nu \alpha_2 |\langle \grad_y C_2 \grad_y f, C_2 f\rangle| \leq \frac{\alpha_2\sqrt{\nu}}{\sqrt{\gamma_2\beta_1}}\left(\nu\gamma_2\|S_2\nabla_yf\|^2 + \beta_1\|C_2f\|^2\right).
\]
Moreover, using Lemma \ref{L21}, we have that
\[
\nu \alpha_2 |\la [\grad_y C_2,C_2] f,\grad_y f\ra |\lesssim \sqrt{\alpha_2} \left(\nu \alpha_2 \|M_2\grad_x f\|^2 +\nu \|\grad_yf\|^2\right).
\]
To control the commutator generated by the advection term, we apply the decomposition of Lemma \ref{L3} and obtain
 \begin{align*}
  \alpha_2 |\langle C_2 f, [C_2 ,u\cdot \grad_x ]f\rangle |& \le \alpha_2 \|C_2 f\| \|C_3f\| +\alpha_2 \|C_2 f\|\|R_3 f\|\\
    &\lesssim \frac{\alpha_2}{\sqrt{\beta_1\beta_2}}\left(\beta_1\|C_2f\|^2 + \beta_2\|C_3f\|^2\right) \\
    & \quad + \frac{\sqrt{\alpha_2}}{\sqrt{\nu\beta_1}}\left(\beta_1\|C_2f\|^2 + \nu\alpha_2\|R_3f\|^2\right).
\end{align*}
Finally, for the $\kappa$-terms, we have
\begin{gather*}
    \kappa \alpha_2 |\la C_2\grad_x f,\grad_x C_2f\ra|  \lesssim \kappa\alpha_2 \|C_2\grad_x f\|\|\grad_xf\| \lesssim \sqrt{\alpha_2}\left(\kappa \alpha_2 \|C_2\grad_x f\|^2 +  \kappa \|\grad_x f\|^2\right)\\
 \kappa\alpha_2  |\la \grad_x f,[\grad_x C_2,C_2]f\ra|\lesssim \alpha_2\kappa \|\grad_xf\|^2.
\end{gather*}

Combining everything and invoking Lemmas \ref{lemma:S2-bound} and \ref{L3}, we obtain,
\begin{equation}
    \label{14}
    \begin{aligned}
  \mel   \ddt\frac{\alpha_2}2\|C_2 f\|^2 +\nu\alpha_2 \|C_2\grad_y f\|^2   +\kappa\alpha_2 \|C_2\grad_x f\|^2  \\
    & \lesssim \left(\frac{\sqrt{\nu}\alpha_2}{\sqrt{\beta_1\gamma_2}}+ \sqrt{\alpha_2}+ \frac{\alpha_2}{\sqrt{\beta_1\beta_2}}+\frac{\sqrt{\alpha_2}}{\sqrt{\nu\beta_1}}  +\alpha_2 \right) \Psi(f).
    \end{aligned}
\end{equation}

\medskip

\noindent
\emph{The $\gamma_2$-term}. The treatment of this term is almost identical to that of the $\alpha_2$-term. We will thus be rather sloppy. First, performing identical computations to those above, we obtain
\begin{align*}
 \mel    \ddt \frac12 \|S_2f\|^2 + \nu \|S_2\grad_y f\|^2 +\kappa \|S_2\grad_x f\|^2\\
 &= -2\nu \la \grad_yS_2 \grad_yf,S_2f\ra +\nu \la [\grad_y S_2,S_2] f,\grad_y f\ra -2\kappa \la S_2\grad_x f,\grad_x S_2f\ra\\
 &\quad +\kappa \la \grad_x f,[\grad_x S_2,S_2]f\ra -\la[S_2,u\cdot \grad_x ]f,S_2f\ra.
\end{align*}

We notice that (modulo signs), the entries of $\grad_y S_2$ are given by those of $C_2$. Hence, 
the $y$-dissipation terms on the right-hand side   can be estimated by
\begin{gather*}
    \nu\gamma_2|\langle \nabla_yS_2\nabla_yf,S_2f\rangle| \leq \frac{\gamma_2}{\alpha_2} \left(\nu\alpha_2\|S_2f\|^2 + \nu\alpha_2\|C_2\nabla_yf\|^2\right),\\
\nu\gamma_2|\la [\grad_y S_2,S_2] f,\grad_y f\ra|\lesssim \frac{\gamma_2}{\sqrt{\alpha_2}}\left(\nu\alpha_2 \|M_2\grad_x f\|^2 + \nu\|\grad_yf\|^2\right),
\end{gather*}
where we have used Lemma \ref{L21} again. Similarly, motivated by the commutator estimate in Lemma \ref{L3-S}, we bound
\[
    \gamma_2|\langle S_2f,[S_2,u\cdot\nabla_x]f\rangle| 
 \leq \frac{\gamma_2}{ \sqrt{\nu\alpha_2\beta_2}} \left(\nu\alpha_2\|S_2f\|^2 + {\beta_2}\|[S_2,u\cdot\nabla_x]f\|^2\right) .
\]
It remains to consider the $\kappa$-terms, which we do as above:
\begin{gather*}
    \kappa \gamma_2 |\la S_2\grad_x f,\grad_x S_2f\ra| 
\lesssim \sqrt{\gamma_2}\left(\kappa \gamma_2 \|S_2\grad_x f\|^2 + \kappa \|\grad_x f\|^2\right)\\
  \kappa\gamma_2| \la \grad_x f,[\grad_x S_2,S_2]f\ra| \lesssim \gamma_2 \kappa \|\grad_x f\|^2.
\end{gather*}

As a last step, we combine these insights and invoke Lemmas \ref{lemma:S2-bound} and \ref{L3-S}, to the effect that
\begin{equation}
    \label{15}
    \ddt \frac{\gamma_2}2\|S_2 f\|^2 + \nu\gamma_2 \|S_2\grad_y f\|^2 + \kappa\gamma_2 \|S_2\grad_x f\|^2 \lesssim \left(\frac{\gamma_2}{\alpha_2}+\frac{\gamma_2}{\sqrt{\alpha_2}} + \frac{\gamma_2}{ \sqrt{\nu\alpha_2\beta_2}}  + \sqrt{\gamma_2} +\gamma_2\right) \Psi(f).
\end{equation}

\medskip

\noindent
\emph{The $\beta_2$-terms}. We will argue similarly as for the $\beta_1$-term and leave out  details whenever the procedure is essentially identical. We will again use the identities $\Delta_y C_2=-C_2$, $\laplace_x C_2=-C_2$ and $\laplace_x C_3 = -C_3$. However, apparently, this identity is not true for the $y$-Laplacian of the third-order commutator. Instead, we have $\laplace_y C_3 = -3C_3 $. 

Following now an equivalent procedure to the computations performed for the $\beta_1$-term, this time using Lemma \ref{L3}, we find the identity  
\begin{align*}
    \ddt \la C_2 f,C_3 f\ra +\|C_3f\|^2 & = 2\nu \la \grad_y C_2f,\grad_y C_3f\ra -2\nu \la C_2\grad_y f,C_3\grad_yf\ra -4\nu \la C_2f,C_3 f\ra\\
    &\quad +2\kappa \la \grad_xC_2 f,\grad_x C_3f\ra -2\kappa \la C_2\grad_x f, C_3\grad_xf\ra - 2\kappa \la C_2f,C_3f\ra\\
    &\quad -\la C_3 f,R_3f\ra  - \la C_2 f,[C_3,u\cdot \grad_x ]f\ra.
\end{align*}

For the estimates, we freely use the fact that $C_3$ contains only derivatives with respect to $x_1$ and $x_2$, and thus, any term involving $C_3$ is controlled by $\grad_x'$. The dissipation terms coming from the Brownian motion are then estimated as follows:
\begin{gather*}
    \nu \beta_2 |\la \grad_y C_2f,\grad_y C_3f\ra |\lesssim \nu \beta_2 \|S_2 f\|\|\grad_x' f\|\lesssim \frac{\sqrt{\nu\beta_2}}{\sqrt{\alpha_2}} \left(\nu\alpha_2 \|S_2f\|^2 +\beta_2 \|\grad_x' f\|^2\right),\\
     \nu \beta_2 |\la C_2\grad_y f,C_3\grad_yf\ra| \lesssim  \frac{\sqrt{\nu\beta_2}}{\sqrt{\alpha_2}} \left(\nu\alpha_2 \|C_2\grad_y f\|^2 +\beta_2 \|\grad_x' f\|^2\right),\\
    \nu\beta_2|\la C_2f,C_3 f\ra|  \le \frac{\sqrt{\beta_2}}{\sqrt{\beta_1}}\left(\beta_1\|C_2f\|^2 +\beta_2\|C_3f\|^2\right).
\end{gather*}
To bound the terms that are generated by the advection term, we have that
\begin{gather*}
    \beta_2 |\la C_3 f,R_3f\ra | \le \frac{\sqrt{\beta_2}}{\sqrt{\nu\alpha_2}} \left(\beta_2 \|C_3f\|^2 +\nu\alpha_2 \|R_3f\|^2\right),\\
    \beta_2 | \la C_2 f,[C_3,u\cdot \grad_x ]f\ra| \le\frac{\sqrt{\beta_2}}{\sqrt{\beta_1}}\left(\beta_1\|C_2f\|^2 +\beta_2\|[C_3,u\cdot \grad_x ]f\|^2\right).
\end{gather*}
Finally, with regard to the $\kappa$-term, we notice the rather brutal bounds
\begin{gather*}
    \kappa \beta_2| \la \grad_xC_2 f,\grad_x C_3f\ra| \lesssim \beta_2\kappa \|\grad_x f\|^2,\\
    \kappa \beta_2|  \la C_2\grad_x f, C_3\grad_xf\ra| \lesssim \kappa \beta_2 \|C_2\grad_x f\| \|\grad_x'\grad_xf\| \le \frac{\beta_2}{\sqrt{\alpha_2\alpha_3}}\left(\kappa \alpha_2\|C_2\grad_xf\|^2 + \kappa\alpha_3\|\grad_x'\grad_xf\|^2\right),\\
    \kappa \beta_2 | \la C_2f,C_3f\ra| \lesssim \beta_2\kappa \|\grad_x f\|^2.
\end{gather*}

Altogether, by the virtue of Lemmas \ref{lemma:S2-bound}, \ref{L5}, and \ref{L3}, \ref{lemma:C4} we have derived the estimate 
\begin{equation}
    \label{16}
    \ddt \la C_2f,C_3f\ra +\|C_3f\|^2 \lesssim \left(\frac{\sqrt{\nu\beta_2}}{\sqrt{\alpha_2}} +  \frac{\sqrt{\beta_2}}{\sqrt{\beta_1}} +\frac{\sqrt{\beta_2}}{\sqrt{\nu\alpha_2}} +\beta_2 +\frac{\beta_2}{\sqrt{\alpha_2\alpha_3}}\right)\Psi(f).
\end{equation}

\medskip
\noindent
\emph{The $\alpha_3$-terms}. To analyse the last term in our energy functional, we study the first entry of $\grad_x'f$. It is readily verified that
\begin{align*}  
\mel \ddt\frac12 \|\partial_1 f\|+\nu \|\grad_y\partial_1 f\|^2 +\kappa\|\grad_x \partial_1 f\|^2 \\
& = -\la \partial_1(u\cdot \grad_x f),\partial_1 f\ra  = -\la \partial_1 u\cdot \grad_xf,\partial_1f\ra = -\la (C_1f)_2,\partial_1f\ra.
\end{align*}
For the only error term on the right-hand side, we notice that
\[
\alpha_3 | \la (C_1f)_2,\partial_1f\ra| \le \alpha_3 \|C_1f\|\|\grad_x' f\| \le \frac{\alpha_3}{\sqrt{\beta_0\beta_2}} \left(\beta_0\|C_1f\|^2 + \beta_2\|\grad_x' f\|^2\right).
\]
The $\partial_2$-derivative can be bounded identically. We thus conclude that
\begin{equation}
    \label{17}
     \ddt\frac{\alpha_3}2\| \nabla_x' f\|^2 +\nu\alpha_3\| \nabla_x'\nabla_yf\|^2  + \kappa\alpha_3\| \nabla_x'\nabla_xf\|^2  
    \lesssim \frac{\alpha_3}{\sqrt{\beta_0\beta_2}} \Psi(f).
 \end{equation}
 
\medskip
\noindent
\emph{Summary}. 
To sum up, we can combine all the estimates from each of the addends in the energy functional \eqref{9}, \eqref{10}, \eqref{11}, \eqref{12}, \eqref{13}, \eqref{14a}, \eqref{14}, \eqref{15}, \eqref{16} and \eqref{17}, to obtain a general energy estimate of the form
\[
\ddt \Phi(f) + \Psi(f) \lesssim \Lambda  \Psi(f),
\]
where $\Lambda= \Lambda(\alpha_i,\beta_i,\gamma_i,\nu,\delta)$ is the sum of all coefficients of $\Psi(f)$. Choosing these as in Lemma \ref{L11} and $\delta=\eps$, we achieve that $\Lambda\lesssim  \eps$. Hence, if $\eps$ is small enough, we deduce 
\[
\ddt \Phi(f) + \frac12\Psi(f) \le0,
\]
and the proof is complete.
\end{proof}

\subsection{Proof of Theorem \ref{T1}}\label{SS1}

In this subsection, we will complete the proof of Theorem \ref{T1}.  We start by establishing a Hardy--Poincar\'e inequality that will be exploited in the proof of  Lemma \ref{L12}.

\begin{lemma}
    \label{L13}
    Suppose that $g:\T^4\to \R$ is a regular function with zero mean,
    \[
    \int_{\T^4}g\, \dd x = 0,
    \]
    with the property that the functions
    \[
    G_i(x) = \frac12\left(g(x)+ g(x+ \pi e_{i-2})\right)
    \]
    are $\pi$-periodic in $x_i$ both for $i=3$ and $i=4$. Then it holds that
    \[
    \|g\|  \lesssim \|\partial_1 g\| + \|\partial_2 g\|+ \|\si_3\partial_3 g\| + \|\si_4\partial_4 g\|.
    \]
\end{lemma}

\begin{proof}The statement is a consequence of the standard 1D Poincar\'e inequality
\begin{equation}
    \label{19}\int_{\T} \left(h-\avint_{\T} h\dd x\right)^2\dd x \lesssim \int_{\T} \left(\frac{\dd}{\dd x} h\right)^2\dd x,
\end{equation}
  and the 1D Hardy--Poincar\'e inequality
    \begin{equation}
    \label{20}\int_{\T} \left(H-\avint_{\T} H\dd x\right)^2\dd x \lesssim \int_{\T} \sin^2(x)\left(\frac{\dd}{\dd x} H\right)^2\dd x,
\end{equation}
for any $\pi$-periodic function $H$.  For the convenience of the reader, we provide the short argument of the latter. Actually, thanks to the $\pi$-periodicity, we may replace the torus by the interval $(0,\pi)$. Moreover, because
\[
\inf_{c\in\R} \int_0^{\pi}(h-c)^2\, \dd x = \int_{\T} \left(h-\frac1{\pi}\int_0^\pi h\dd x\right)^2\dd x,
\]
and by the symmetry of the problem it is enough to prove the slightly stronger statement
\[
\int_0^{\pi/2} \left(h-h(\pi/2)\right)^2 \dd x \lesssim \int_0^{\pi/2} \sin^2(x)\left(\frac{\dd}{\dd x} h\right)^2\dd x.
\]
This estimate can be established as follows: Noticing that $1/\sqrt{2}\le  \cos(x/2)= 2\frac{\dd}{\dd x}\sin(x/2)$ for $x\in(0,\pi/2)$, we observe that
\[
\begin{split}
\int_0^{\pi/2} \left(h-h\left(\frac{\pi}{2}\right)\right)^2\dd x & \lesssim \int_0^{\pi/2}\frac{\dd}{\dd x}\sin\left(\frac{x}{2}\right) \left(h-h\left(\frac{\pi}{2}\right)\right)^2\dd x  \\
& = -2 \int_0^{\pi/2} \sin\left(\frac{x}{2}\right)\left(h-h\left(\frac{\pi}{2}\right)\right)\frac{\dd}{\dd x} h\dd x,
\end{split}
\]
because boundary terms disappear in the integration by parts. Therefore, using the Cauchy--Schwarz inequality and observing also that $\sin(x/2)\lesssim \sin(x)$ for any $x\in(0,\pi/2)$, the desired estimate follows.

It remains to indicate how the two inequalities \eqref{19} and \eqref{20} imply the statement of the lemma.  We use a telescope sum argument to write
\begin{align*}
\mel g(x_1,x_2,x_3,x_4) - \avint_{\T^4} g(\tilde x_1,\tilde x_2,\tilde x_3,\tilde x_4)\,\dd (\tilde x_1,\tilde x_2,\tilde x_3,\tilde x_4)\\
& = g(x_1,x_2,x_3,x_4) - \avint_{\T} g(\tilde x_1,x_2,x_3,x_4)\, \dd \tilde x_1\\
&\quad + \avint_{\T}\left( g(\tilde x_1,x_2,x_3,x_4) - \avint_{\T} g(\tilde x_1,\tilde x_2,x_3,x_4)\, \dd \tilde x_2\right)\dd \tilde x_1\\
&\quad +\avint_{\T^2} \left(g(\tilde x_1,\tilde x_2,x_3,x_4)-\avint_{\T}g(\tilde x_1,\tilde x_2,\tilde x_3,x_4)\, \dd \tilde x_3\right)\dd (\tilde x_1,\tilde x_2)\\
&\quad+ \avint_{\T^3} \left(g(\tilde x_1,\tilde x_2,\tilde x_3,x_4)-\avint_{\T}g(\tilde x_1,\tilde x_2,\tilde x_3,\tilde x_4)\, \dd \tilde x_4\right)\dd (\tilde x_1,\tilde x_2,\tilde x_3).
\end{align*}
By a shift in the variables, we may replace $g$ in the last two lines by the $\pi$-periodic functions $G_3$ and $G_4$, respectively. To be more specific, we have, for instance, for any fixed $\tilde x_2, x_3$ and $x_4$,
\begin{align*}
   \mel  \avint_{\T} \left(g(\tilde x_1,\tilde x_2,x_3,x_4)-\avint g(\tilde x_1,\tilde x_2,\tilde x_3,x_4)\dd \tilde x_3 \right)\dd \tilde x_1\\
  &  =   \avint_{\T} \left( g(\tilde x_1+\pi,\tilde x_2,x_3,x_4)-\avint   g(\tilde x_1+\pi,\tilde x_2,\tilde x_3,x_4) \dd \tilde x_3  \right) \dd \tilde x_1 \\
  & = \avint_{\T} \left(G_3(\tilde x_1,\tilde x_2,x_3,x_4)-\avint G_3(\tilde x_1,\tilde x_2,\tilde x_3,x_4)\dd \tilde x_3 \right)\dd \tilde x_1.
\end{align*}
After integration, the first two terms can be estimated using the Poincar'e inequality \eqref{19} with respect to $x_1$ and $x_2$. The last two terms are bounded with the help of the Hardy--Poincar\'e inequality \eqref{20} applied in $x_3 $ and $x_4$. This proves the lemma.
\end{proof}

We now turn to the proof of Lemma \ref{L12}, which then allows us to estimate the dissipation term from below by the Lyapunov functional. 

\begin{proof}
    [Proof of Lemma \ref{L12}]
    We start by noticing that the indefinite $\beta_i$ terms can be dropped in the definition of $\Phi(f)$. Indeed, recalling that $C_0=\grad_y$, we have for any index $i\in\{0,1,2\}$
    \[
    \beta_i |\la C_i f,C_{i+1}f\ra| \le  \frac{\beta_i}{\sqrt{\alpha_i\alpha_{i+1}}} \left( \frac{\alpha_i}2 \|C_if\|^2 + \frac{\alpha_{i+1}}2\|C_{i+1}f\|^2\right).
    \]
    In the case $i=2$, we may trivially estimate the third commutator $C_3f$ by $\grad_x'f$. For any $\eps\le1$ in the hypothesis of Lemma \ref{L11}, the prefactors are bounded uniformly by $\sqrt{\eps}$, and hence, these terms are controlled by the remaining terms of $\Phi(f)$, if $\eps$ is chosen small enough.
    
    Therefore, using Lemmas \ref{L9}, \ref{lemma:S2-bound} and \ref{L5} to control $S_1f$, $S_2f$, and $M_3\nabla_x f$, we find that
    \begin{equation}
        \label{21}
     \Phi(f) \lesssim \|f\|^2 + \left(\frac{\alpha_0}{\nu} + \frac{\alpha_1}{\beta_0}+\frac{\gamma_1}{\nu\alpha_1}+\frac{\alpha_2}{\beta_1} +\frac{\gamma_2}{\nu\alpha_2}+ \frac{\alpha_3}{\beta_2} +  \frac1{\nu}+\delta\right) \Psi(f).
      \end{equation}
We choose $\delta=\eps<1$. The remaining terms multiplying the dissipation term $\Psi(f)$ are all of the order $1/\nu$ if the coefficients scale as in \eqref{23}.

In order to control the $L^2$ norm of $f$ by the dissipation term, we have to apply the Hardy--Poincar\'e inequality from Lemma \ref{L13}. It is applicable, because $f$ has zero mean with respect to the $x$ variable. Indeed, because $f$ has zero mean initially and because the mean satisfies the heat equation
    \[
    \frac{\partial}{\partial t}\avint_{\T^4}f\, \dd x  = \nu\laplace_y \avint_{\T^4} f\, \dd x,
    \]
    it remains zero during the evolution.
    Hence, thanks to Lemma \ref{L13} we find that
    \[
    \|f\|\lesssim \|\grad_x'f\| + \|M_2\nabla_xf\|.
    \]
Invoking now the bounds in Lemmas \ref{lemma:S2-bound} and \ref{L5}, we obtain that
    \[
    \|f\|\lesssim \left(\frac1{\nu\alpha_2}+\frac1{\beta_2}\right)\Psi(f).
    \]
This time, the prefactor   of the dissipation is of the order $\nu^5$ by the scaling in \eqref{23}, and it dominates thus the ones that we discussed above. Inserting this estimate into \eqref{21} gives thus the desired statement.
\end{proof}

Next, we prove that the equivalence between the energy functional $\Phi$ and the $\Ha^1$ norm.

\begin{proof}
    [Proof of Lemma \ref{L14}]
The proof is in parts similar to the previous one. Arguing as before, we notice that we can drop the indefinite $\beta_i$ terms in $\Phi(f)$ provided that  $\eps$ in the hypothesis of Lemma \ref{L11} is sufficiently small. First, for the bound $\Phi(f)\lesssim \|f\|_{\mathcal{H}^1}^2$ we simply use the trivial pointwise estimates
\[
|C_1f|+|S_1f|+|C_2f|+|S_2f|\lesssim  |\grad_x'f| + |\si_4\partial_3 f| + |\si_3\partial_4 f| + |\si_3\partial_3f| + |\si_3\partial_4 f| .
\]
Towards the reciprocal bound $\|f\|_{\mathcal{H}^1}^2\lesssim \Phi(f)$, we make use of the pointwise lower bounds from Lemmas \ref{lemma:bound-D1} and \ref{L1},
\[
|C_1f|+|S_1f| + |C_2f|+|S_2f| \gtrsim |\si_4\partial_3 f|+|\si_3\partial_4f| +|\si_3\partial_3 f| +|\si_4 \partial_4 f|.
\]
This establishes the equivalence.
\end{proof}

We finally ensure the compatibility of the conditions on the coefficients.

\begin{proof}[Proof of Lemma \ref{L11}]
We focus on  coefficients of the form
\[
\alpha_i = a_i\nu^{-2i}, \quad \beta_i = b_i\nu^{-(2i+1)}, \quad \gamma_i = c_i\nu^{-2i}, \quad ,
\]
with $a_i,b_i,c_i>0$ of order $1$ in $\nu$. Taking into account all the conditions that are needed to be satisfied in terms of $\eps>0$, we can assume that the coefficients in Latin letters are defined as a suitable power of $\eps$, namely,
\[
a_i = \eps^{x_i}, \quad b_i = \eps^{y_i}, \quad c_i = \eps^{z_i}.
\]
With this nomenclature we can translate the conditions for the Greek letters from Lemma \ref{L11} to conditions for $x_i$, $y_i$ and $z_i$. For instance, the condition $\alpha_0^2\ll_\eps \nu\beta_0$ can be written as
\[
\alpha_0^2\leq \eps\nu\beta_0 \quad \Leftrightarrow \quad \eps^{2x_0} \leq \eps^{1+y_0} \quad \Leftrightarrow \quad 2x_0\geq 1+y_0.
\]
An analogous procedure for all the conditions stated in the lemma yields, on the one hand, that the sequence \eqref{22} is equivalent to imposing
\begin{gather*}
    x_0\geq 1, \\
    y_0\geq 1+x_0, \quad x_1\geq 1+y_0, \quad z_1\geq 1+x_1, \quad y_1\geq 1+z_1, \\
    x_2\geq 1+y_1, \quad z_2\geq 1+x_2, \quad x_3\geq 1+z_2, \quad y_2\geq 1+x_3.
\end{gather*}
On the other hand, the rest of the conditions stated in Lemma \ref{L11} are equivalent to the following set of requirements.
\begin{gather*}
    2x_0\geq 1+y_0, \quad 2y_0\geq 1+y_1, \quad 2x_1\geq 1+y_0+y_1, \quad 2x_2\geq 1+y_1+y_2, \\
    2y_0\geq 1+x_0+x_1, \quad 2y_1\geq 1+x_1+x_2, \quad 2y_2\geq 1+x_2+x_3, \quad 2y_1\geq 1+y_0+y_2, \\
    2x_1\geq 1+y_0+z_1,  \quad 2z_1\geq 1+x_1+x_2, \quad 2z_2\geq 1+x_2+y_2.
\end{gather*}
This is a solvable system with more than one solution. In particular, the following choice of parameters gives a compatible solution:
\begin{gather*}
    x_0 = 256, \quad x_1=448, \quad x_2=488,\quad x_3=492, \\
    y_0 = 384, \quad y_1 = 480, \quad y_2=493, \\
    z_1=472, \quad z_2=491.
\end{gather*}
This completes the proof of the lemma.
\end{proof}

\section{Mixing and enhanced dissipation}\label{s:mixing}

In this section we elaborate on the strategy presented in Section \ref{s:strategy} and explain  in more detail how to obtain our statement about the exponential decay of correlations from the hypocoercivity result in Theorem \ref{T1}. This will eventually lead to the proofs of Propositions \ref{prop:mix} and \ref{prop:diss}. The core ideas presented in this section are mostly already available in the existent literature, see e.g.\ \cites{BBPS22,BlumenthalCotiZelatiGvalani23}. Nonetheless, we include them in detail here in our specific setting for completion.

To begin with, we recall from Section \ref{s:strategy} that we denote by  $X_t$ and $Z_t$ the stochastic flows introduced in \eqref{flow}, and by  $Y_t$ the one from \eqref{heat}. Moreover, we recall  that on the one hand $\EX_\nu$ denotes the expectation operator with respect to the noise induced by $\tilde B_t$, and on the other hand $\EX_\kappa$ denotes the expectation with respect to the noise induce by $B_t$ or $\hat B_t$. We follow an analogous nomenclature for the probability measures $\Prob_\nu$ and $\Prob_\kappa$ respectively.

Without further ado, we introduce the first result required for the proofs of Propositions \ref{prop:mix} and \ref{prop:diss}. The following lemma is a first step in the formalization of the Borel--Cantelli argument that was sketched in Section \ref{s:strategy}.

\begin{lemma}\label{lemma41}
Assume that $\theta_0$ has zero mean and the initial distribution $\rho_0$ satisfies  \eqref{rho0}. Assume that the solution $f$ of equation \eqref{hypo} satisfies the estimate
\[
\|f(t)\|_{\Ha^1(\T^2\times\T^2\times\T^2)} \lesssim \e^{-\lambda t}\|f_0\|_{\Ha^1(\T^2\times\T^2\times\T^2)},
\]
for any $t\ge 0$, with $\lambda>0$ independent of $\kappa$ and with initial configuration of the form $f_0=\theta_0\otimes\theta_0\otimes\rho_0$.  Then, for any $h\in L^2(\T^2)$ and any bounded solution $\theta(t,\cdot)$  to \eqref{eq:AD} starting from $\theta_0$, we have the estimate
\[
\EX_\nu\la h,\theta(t)\ra^2 \lesssim \|h\|^2_{L^2}\|\rho_0\|_{W^{1,\infty}}\|\theta_0\|_{L^{\infty} }\|\theta_0\|_{H^1}\e^{-\lambda t},
\]
for any $t\geq 0$. 
\end{lemma}

\begin{proof}
For the sake of a clearer presentation we go back to the notation used in Section \ref{s:strategy}, namely we write $(x,z)\in\T^2\times\T^2$ for  $x\in\T^4 $. Arguing as before, we compute the expectation of $\la h,\theta(t)\ra^2$ by factorising the square using double variables and by invoking the Feynman--Kac representation of solutions to the passive scalar equation \eqref{eq:AD}, 
\[
\begin{split}
\mel \EX_\nu\la h,\theta(t)\ra^2 \\
& = \EX_\nu \left(\int_{\T^2} h(x)\theta(t,x)\dd x\right)^2 = \EX_\nu \iint_{\T^2\times\T^2} h(x)h(z)\theta(t,x)\theta(t,z)\dd x\dd z\\
& = \EX_\nu \iint_{\T^2\times\T^2} h(x)h(z)\EX_\kappa \theta_0(X_t^{-1}(x;\{Y_s(y)\}_{s\in[0,t]}))\EX_\kappa \theta_0(Z_t^{-1}(z;\{Y_s(y)\}_{s\in[0,t]}))\dd x\dd z.
\end{split}
\]
Now we consider $\rho_0$ as in the claim of the lemma and fix one realisation of the stochastic process $Y_t$, which is volume preserving with $\Prob_\nu$-probability $1$. Therefore we have that
\[
\int_{\T^2}\rho_0(Y_t^{-1}(y))\dd y = 1\quad \Prob_\nu-\textrm{almost surely},
\]
and, thus,  we can plug it inside the expectation $\EX_\nu$ as a multiplicative factor. More precisely, we can write
\[
 \EX_\nu\la h,\theta(t)\ra^2 
 = \iiint_{\T^2\times \T^2\times\T^2} h(x)h(z)f(t,x,z,y)\dd x\dd z\dd y,
\]
where we have set
\[
f(t,x,z,y) = \EX_\nu \left[\EX_\kappa \theta_0(X_t^{-1}(x;\{Y_s(y)\}_{s\in[0,t]}))\EX_\kappa \theta_0(Z_t^{-1}(z;\{Y_s(y)\}_{s\in[0,t]}))\rho_0(Y_t^{-1}(y))\right].
\]
Using the Feynman--Kac formula, we find that the function $f$ satisfies precisely the (hypo-) elliptic PDE studied in Section \ref{s:hypocoercivity}. We refer to Appendix \ref{s:feyman-kac} for details. Following the notation of this proof, where we consider $(x,z,y) \in \T^2 \times \T^2 \times \T^2$ as the spatial variables, this PDE takes the form  
\[
\partial_tf + v(x,y)\cdot\nabla_xf + v(z,y)\cdot\nabla_zf = \kappa(\Delta_xf + \Delta_zf) + \nu\Delta_yf.
\]
 To sum up, we can use Hölder's inequality and the assumed decay on $f$ to estimate
\[
\EX_\nu\la h,\theta(t)\ra^2 \leq \|h\otimes h\otimes 1\|_{L^2}\|f(t)\|_{L^2} \leq \|h\otimes h\otimes 1\|_{L^2}\|f(t)\|_{\Ha^1}\lesssim \|h \|_{L^2}^2\|f_0\|_{\Ha^1}\e^{-\lambda t},
\]
for any $t\ge0$. Using the particular form of our initial datum $f_0(x,z,y) = \theta_0(x)\theta_0(z)\rho_0(y)$, we conclude that
\[
\|f_0\|_{\Ha^1(\T^2\times\T^2\times\T^2)} \lesssim \|\rho_0\|_{W^{1,\infty}(\T^2)}\|\theta_0\|_{L^\infty(\T^2)}\|\theta_0\|_{H^1(\T^2)}.
\]
Hence we arrive to the claim of the lemma.
\end{proof}

A small variant of the previous proof yields an averaged mixing result in which the initial data is taken in $L^2(\T^2)$.
\begin{lemma}\label{prop:mixing-L2rhs}
    Assume that $\theta_0$ has zero mean and the initial distribution $\rho_0$ satisfies  \eqref{rho0}. Assume that the solution $f$ of equation \eqref{hypo} satisfies the estimate
\[
\|f(t)\|_{\Ha^1(\T^2\times\T^2\times\T^2)} \lesssim \e^{-\lambda t}\|f_0\|_{\Ha^1(\T^2\times\T^2\times\T^2)},
\]
for any $t\ge 0$, with $\lambda>0$ independent of $\kappa$ and with initial configuration of the form $f_0=h\otimes h\otimes\rho_0$ with $h\in H^1(\T^2)$.  Then, for any   solution $\theta(t,\cdot)$  to \eqref{eq:AD} starting from $\theta_0\in L^2(\T^2) $, and any $p\in[1,\infty)$, we have the estimate
\[
\EX_\nu\|\theta(t)\|_{\dot H^{-1}}^p  \lesssim  \|\theta_0\|_{L^{2} }^p  \e^{-\lambda_p t},
\]
for some $\lambda_p \in (\lambda/3,\lambda]$ and  any $t\geq 0$. In particular, it holds that
\[
\|\theta(t)\|_{\dot H^{-1}} \le C(\theta_0) \|\theta_0\|_{L^2}e^{-\frac{1}{2}\lambda_p t}
\]
for any $t\ge 0$, where $C(\theta_0)$ is a positive (random) constant depending on the initial datum $\theta_0$ and satisfying $\EX_{\nu}C(\theta_0)^q <\infty$ for any $q \in [1,\infty)$.
\end{lemma}

Actually, our proof shows that $\lambda_p=\lambda$ for $p>3$.

\begin{proof} Arguing similarly as before, we write 
\begin{align*}
  \mel   \EX \la h,\theta(t)\ra^2\\
  & =  \EX_\nu \EX_\kappa\EX_\kappa \iint_{\T^2\times\T^2} h(x)h(z)  \theta_0(X_t^{-1}(x;\{Y_s(y)\}_{s\in[0,t]}))  \theta_0(Z_t^{-1}(z;\{Y_s(y)\}_{s\in[0,t]}))\dd x\dd z\\  
  & =  \EX_\nu \iint_{\T^2\times\T^2}\EX_\kappa  h(X_t (x;\{Y_s(y)\}_{s\in[0,t]}))\EX_\kappa h(Z_t(z;\{Y_s(y)\}_{s\in[0,t]})) \theta_0(x)  \theta_0(z)\dd x\dd z.
\end{align*}    
Again, we may smuggle in the initial density function $\rho_0$. Setting
\[
g(t,x,z,y) = \EX_\nu \left[\EX_\kappa h(X_t(x;\{Y_s(y)\}_{s\in[0,t]}))\EX_\kappa h(Z_t(z;\{Y_s(y)\}_{s\in[0,t]}))\rho_0(Y_t^{-1}(y))\right],
\]
we then find
\[
\EX \la h,\theta(t)\ra^2 =\iiint_{\T^2\times\T^2 \times \T^2} g(t,x,z,y)\theta_0(x)\theta_0(z)\dd x\dd x\dd y.
\]
Using the Feynman--Kac formula that we derived in Appendix \ref{s:feyman-kac}, we observe that $g$ solves the two-point PDE
\[
\partial_t g - v(x,y)\cdot \grad_x g - v(z,y)\cdot\grad_z g = \kappa(\Delta_x g + \Delta_z g) +\nu \Delta_y g.
\]
The equation can be transformed into our two-point PDE by a rotation by 180 degrees. Therefore, using H\"older's inequality and the assumption of the proposition, we find that
\[
\EX_{\nu} \la h,\theta(t)\ra^2 \lesssim \|\theta_0\|_{L^2}^2\|g(t)\|_{L^2}\lesssim \|\theta_0\|_{L^2}^2 \|h\|_{L^{\infty}}\|h\|_{H^1}\e^{-\lambda t},
\]
for any $t\ge0$. Choosing now $h=\e^{ik\cdot x}$, the latter becomes
\[
\EX_{\nu} \la \e^{ik\cdot x},\theta(t)\ra^2 \lesssim  |k|\|\theta_0\|_{L^2}^2  \e^{-\lambda t}.
\]
Multiplication by $|k|^{-2s}$ and summation over $k$ then gives
\[
\EX_{\nu} \|\theta(t)\|_{\dot H^{-s}}^2   = \EX_{\nu} \left(\sum_{k\in 2\pi \Z^2} |k|^{-2s} \la \e^{ik\cdot x},\theta(t)\ra^2\right) \lesssim \left(\sum_{k\in2\pi \Z^2}|k|^{1-2s}\right)\|\theta_0\|_{L^2}^2 \e^{-\lambda t}\lesssim \|\theta_0\|_{L^2}^2 \e^{-\lambda t},
\]
provided that $s>3/2$. It remains to recall that we always have that $\|\theta(t)\|_{L^2}\le \|\theta_0\|_{L^2}$ by the energy balance. Hence, we may interpolate
\[
 \EX_{\nu} \|\theta(t)\|_{\dot H^{-1}}^p  \le \left(\EX_{\nu} \|\theta(t)\|_{\dot H^{-\frac{p}2}}^2\right) \|\theta_0\|_{L^2}^{p-2}\lesssim \|\theta_0\|_{L^2}^p \e^{-\lambda t} ,
\]
for any $p >3$. The estimate can be extended to the range $p\in[1,3]$ with the help of Jensen's inequality, which actually reduces the exponential rate by a numerical factor.

We now turn to the second statement. It is enough to consider the case $t=n\in\N$ because the general case can then be obtained via the energy balance $\|\theta(t)\|_{L^2} \le \|\theta_0\|_{L^2}$. We define
\[
C(\theta_0) = \sup_{n\in \N} \frac{\|\theta(n)\|_{\dot H^{-1}}}{\|\theta_0\|_{L^2}} \e^{\frac{\lambda n}2},
\]
so that we have the trivial estimate
\[
\|\theta(n)\|_{\dot H^{-1}} \le C(\theta_0) \|\theta_0\|_{L^2}\e^{-\frac{\lambda n}2}.
\]
The moment bounds follow from the previous estimate: for $p>3$, we have that
\[
\EX_{\nu} |C(\theta_0)|^p \le \sum_{n\in \N} \EX_{\nu} \frac{\|\theta(n)\|_{\dot H^{-1}}^p}{\|\theta_0\|_{L^2}^p}\e^{\frac{\lambda p n}2} \lesssim \sum_{n\in \N} \e^{-\frac{\lambda pn}2} <\infty,
\]
and for smaller values of $p$, the statement follows via Jensen's inequality. This completes the proof.
\end{proof}

\begin{proposition}\label{prop41}
Let $q>0$ and $s>1$ be fixed. Under the assumptions of Lemma \ref{lemma41}, there exists a $\kappa$-independent and deterministic constant $\gamma>0$, and a $\Prob_\nu$-a.s.\ finite random function $\hat C>0$ such that for any two mean-zero functions $\theta_0, h\in H^s(\T^2)$, the following estimate holds true
\[
\left|\EX_{\kappa}\int_{\T^2} \theta_0(x)h(X_n(x))\dd x \right| \leq \hat C \|h\|_{H^s}\|\theta_0\|_{H^s}\e^{-\gamma n},
\]
for all $n\in\N$. Moreover, $\EX_\nu[\hat C^q]<\infty$.
\end{proposition}

We will see in the proof of Proposition \ref{prop:mix} that the result extends to any $s>0$ via duality and interpolation arguments. Moreover, it  extends to any time $t\ge0$ via stability estimates.

The idea behind this proposition revolves around the Borel--Cantelli argument that was sketched in Section \ref{s:strategy}, and goes back to the original approach by Bedrossian, Blumenthal and Punshon-Smith \cite{BBPS22}. For a more accessible proof where the decay, obtained via Harris Theorem, occurs in $L^\infty$, see \cite{BlumenthalCotiZelatiGvalani23}*{Proposition 4.6}. Since in our case the decay happens with $L^2$ norms and there are small technical differences, we present here a proof of the result that is tailor made for our needs.

\begin{proof}
Let $\{\e^{ik \cdot x}/\sqrt{2\pi}: k\in \Z^2\}$ be the orthonormal Fourier basis for $L^2(\T^2)$, which we equip with the complex inner product $\la\cdot ,\cdot\ra$.
Fix a positive number $\gamma>0$, two wave numbers $k,k'\in i\Z^2\setminus\{0\}$, and define the random variable
\[
N_{k,k'} = \max\left\lbrace n\in \N :  \left|\EX_{\kappa}\la  \e^{ik'\cdot X_n(x)}, \e^{ik\cdot x}\ra \right| > 2\pi \e^{-\gamma n}\right\rbrace.
\]

On the one hand, the probability for the random variable $N_{k,k'}$ to be larger that some positive number $\ell>0$ can be estimated via Chebyshev's inequality,
\begin{align*}
\Prob_\nu[N_{k,k'}>\ell] &\leq \sum_{n>\ell} \Prob_\nu \left[ \left| \EX_{\kappa} \la  \e^{ik'\cdot X_n(x)} ,\e^{ik\cdot x}\ra \right| >2\pi  \e^{-\gamma n} \right] \\
&\leq \frac1{(2\pi)^2}\sum_{n>\ell} \e^{2\gamma n} \EX_\nu \left|\EX_{\kappa}\la  \e^{ik'\cdot X_n(x)} ,\e^{ik\cdot x}\ra\right|^2.
\end{align*}
We let $\tilde \theta_k$ denote the solution to the passive scalar equation \eqref{eq:AD} with initial datum $\e^{ik\cdot x}$, so that $\tilde \theta_k(n,x) = \EX_{\kappa} \e^{ik\cdot X_n^{-1}(x)}$. Then, by  a change of variables, we may rewrite the previous estimate as
\[
\Prob_{\nu}[N_{k,k'}>\ell] \le \frac1{(2\pi)^2} \sum_{n>\ell} \EX_\nu\left|\la \e^{ik'\cdot x}, \tilde \theta_k(n,x)\ra\right|^2.
\]

Noticing that since the basis elements $\e^{ik\cdot x}$ are  in $H^1(\T^2)$, we can apply Lemma \ref{lemma41} so that
\[
\EX_\nu\left|\la \e^{ik'\cdot x}, \tilde \theta_k(n,x)\ra\right|^2 \lesssim \|\e^{ik'\cdot x}\|_{L^2}\|\rho_0\|_{W^{1,\infty}}\|\e^{ik\cdot x}\|_{L^{\infty}} \|\e^{ik\cdot x}\|_{H^1} \e^{-\lambda n} \lesssim 
|k|\e^{-\lambda n},
\]
where we used that  the $H^1$ of $\e^{ik\cdot x}$ is controlled by $|k|$. A combination of the previous two estimates gives
\begin{equation}\label{tail}
\Prob_\nu[N_{k,k'}>\ell] \lesssim |k|\sum_{n>\ell} \e^{-(\lambda-2\gamma)n} \lesssim |k|\e^{-(\lambda-2\gamma)\ell}.
\end{equation}
For $\gamma<\lambda/2$, then the right-hand side is summable and we can apply the Borel--Cantelli lemma, which implies
\[
\Prob_\nu\left[\limsup_{\ell\to\infty} N_{k,k'}>\ell\right] = 0.
\]
Therefore, $N_{k,k'}$ is $\Prob_\nu-$almost surely finite.

On the other hand, by definition of the random variable $N_{k,k'}$ we get that
\[
\left| \EX_{\kappa}\la  \e^{ik'\cdot X_n(x)}, \e^{ik\cdot x}\ra \right| \leq 2\pi \e^{\gamma( N_{k,k'}-n)},
\]
for any $n\in\N$. Now if we use the Fourier basis to write the functions $h$ and $\theta_0$ such that
\[
h(x) = \frac1{\sqrt{2\pi}}\sum_{k'\in  \Z^2\setminus\{ 0\}} \hat h(k')\e^{ik'\cdot x}, \quad \theta_0(x) = \frac1{\sqrt{2\pi}} \sum_{k\in  Z^2\setminus\{0\}} \hat \theta_0(k)\e^{ik\cdot x},
\]
the latter observation implies that
\begin{equation}\label{eq42}
\left|\la  \theta_0 ,\EX_{\kappa}h\circ X_n \ra \right| \lesssim \e^{-\gamma n}\sum_{k,k'}|\hat h(k')||\hat\theta_0(k)|\e^{\gamma N_{k,k'}}.
\end{equation}

Next in order, let us define the following random variable,
\[
K = \max\left\lbrace |k|,|k'|\geq 1: \e^{\gamma N_{k,k'}}>|k|^\alpha|k'|^{\alpha'}\right\rbrace,
\]
for some coefficients $\alpha,\alpha'>0$ to be chosen later. In a similar fashion to how we argued for $N_{k,k'}$, and using the tail estimate \eqref{tail}, for any $\ell>0$ we find that
\[
\begin{split}
\Prob_\nu[K>\ell] & \leq \sum_{|k|,|k'|>\ell} \Prob_\nu\left[ \e^{\gamma N_{k,k'}}>|k|^\alpha|k'|^{\alpha'} \right] = \sum_{|k|,|k'|>\ell} \Prob_\nu\left[ N_{k,k'}>\frac{1}{\gamma}\left(\alpha\log|k|+\alpha'\log|k'|\right) \right] \\
& \lesssim \sum_{|k|,|k'|>\ell}|k|^{1-\frac{\alpha}{\gamma}(\lambda-2\gamma)}|k'|^{-\frac{\alpha'}{\gamma}(\lambda-2\gamma)} \lesssim \left(\sum_{|k|>\ell}|k|^{1-\frac{\alpha}{\gamma}(\lambda-2\gamma)} \right)\left(\sum_{|k'|>\ell}|k'|^{-\frac{\alpha'}{\gamma}(\lambda-2\gamma)}\right) \\
&\lesssim \ell^{3-\frac{\alpha}{\gamma}(\lambda-2\gamma)}\ell^{2-\frac{\alpha'}{\gamma}(\lambda-2\gamma)} = \ell^{5-\frac{1}{\gamma}(\alpha+\alpha')(\lambda-2\gamma)},
\end{split}
\]
provided that the infinite sums are convergent, which means
\begin{equation}\label{cond1}
\alpha> \frac{3\gamma}{\lambda-2\gamma} \quad \text{and}\quad \alpha'> \frac{2\gamma}{\lambda-2\gamma}.
\end{equation}
Additionally, in order to apply yet again the Borel--Cantelli lemma, we need the power of $\ell$ to be sufficiently small so that $\Prob_\nu[K>\ell]$ is summable in $\ell$, that is, we need the condition
\begin{equation}\label{cond2}
\alpha+\alpha' > \frac{6\gamma}{\lambda-2\gamma}.
\end{equation}
Therefore thanks to the Borel--Cantelli lemma, we obtain that for any pair $(\alpha,\alpha')$ satisfying conditions \eqref{cond1} and \eqref{cond2}, the random variable $K$ is $\Prob_\nu-$almost surely finite, so that we can define another $\Prob_\nu-$almost surely finite random variable
\[
\hat C = \max_{|k|,|k'|\leq K} \e^{\gamma N_{k,k'}}.
\]
that satisfies the estimate
\[
\e^{\gamma N_{k,k'}}\leq \hat C|k|^\alpha|k'|^{\alpha'},
\]
for all $k,k'\in  \Z^2\setminus\{0\}$. We know that all $q$-moments of this random variable are finite, i.e., $\EX_\nu[\hat C^q]<\infty$ for all $q>0$, thanks to an argument that can be found in \cite{BBPS22}*{Section 7.3} and that we will skip in this proof. Therefore, going back to \eqref{eq42}, putting everything together, and using the Cauchy--Schwarz inequality to recover the $H^s$ norms, we can write
\[
\begin{split}
\left|\EX_{\kappa}\int_{\T^2} \theta_0(x)h(X_n(x))\dd x \right| & \lesssim \hat C\e^{-\gamma n}\sum_{k,k'}|\hat h(k')||\hat\theta_0(k)||k|^\alpha|k'|^{\alpha'} \\
& \lesssim \hat C\e^{-\gamma n} \left(\sum_k |\hat\theta_0(k)||k|^\alpha \right)\left(\sum_{k'} |\hat h(k')||k'|^{\alpha'} \right) \\
& \lesssim \hat C \|h\|_{H^s}\|\theta_0\|_{H^s}\e^{-\gamma n},
\end{split}
\]
that holds true for indices $s>0$ such that $s>\max\{\alpha+1,\alpha'+1\}$.  Notice from \eqref{cond1} and \eqref{cond2} that $\alpha$ and $\alpha'$ can be chosen arbitrarily close to $0$ at the expense of a worse rate of convergence $\gamma>0$, therefore our argument shows that $s$ can be chosen arbitrarily close to but larger than~$1$. 
\end{proof}

With this result we can now proceed with the proof of Proposition \ref{prop:mix}.

\begin{proof}[Proof of Proposition \ref{prop:mix}]
We fix $s>1$ arbitrarily. Putting together the results from Lemma \ref{lemma41} and Proposition \ref{prop41}, we find that exponential decay in $\Ha^1(\T^2)$ of solutions $f$ to equation \eqref{hypo} implies that for any $h\in H^s(\T^2)$, all  mean-free solutions $\theta$ to the passive scalar equation \eqref{eq:AD} with initial datum $\theta_0\in H^s(\T^2)$ satisfy the estimate
\[
\left|\int_{\T^2} \theta(n,x)h(x)\dd x \right| \leq \hat C\|h\|_{\dot H^s}\|\theta_0\|_{\dot H^s}\e^{-\gamma n},
\]
for any discrete times $n\in\N$, where $\gamma$ is a constant independent of $\kappa$.
Arguing by duality, we immediately see that the latter implies a norm estimate
\[
\|\theta(n)\|_{\dot H^{-s}} = \sup_{\|h\|_{\dot H^s}=1}\int_{\T^2} \theta(t,x)h(x)\dd x  \leq \hat C\|\theta_0\|_{\dot H^s}\e^{-\gamma n},
\]
for any $n\in\N$. In view of the trivial monotonicity estimate $\|\theta(n)\|_{L^2}\le \|\theta_0\|_{L^2}$, which corresponds to the case $s=0$, real interpolation allows us to extend this estimate to any $s>0$. In fact, if we denote by $\gamma$ the exponent corresponding to $s=1$, then the interpolation argument shows that for general $s$, the exponential decay proceeds with rate $s\gamma$.

Finally, to attain the result for all positive times, let us argue for simplicity for the case $s=1$, although the general case $s>0$ follows analogously. We can fix $\tau\in[0,1)$ and define $t=\tau+n$ so that the result for discrete time gives us the estimate
\[
\|\theta(t)\|_{\dot H^{-1}} \leq \hat C\e^\gamma \|\theta(\tau)\|_{\dot H^1}\e^{-\gamma t}.
\]
Finally a standard stability estimate for $\nabla\theta$ in \eqref{eq:AD} implies that
\[
\|\theta(\tau)\|_{\dot H^1} \lesssim \|\theta_0\|_{\dot H^1} \e^{\int_0^\tau \|\nabla v(s)\|_{L^\infty}\dd s} \lesssim \|\theta_0\|_{\dot H^1}, \quad \forall \tau\in [0,1),
\]
because our vector field $v$ is smooth. Now redefining $\hat C$ so that it includes the factor $\e^{\gamma}$ as well as the constants absorbed by $\lesssim$ in this last estimate, we find the claim of the proposition.
\end{proof}

Next in order, we prove that uniform-in-diffusivity mixing implies enhanced dissipation with exponential rate independent of $\kappa$

\begin{proof}[Proof of Proposition \ref{prop:diss}]
Let $\kappa>0$, and assume that we have mixing uniform in diffusivity,
\begin{equation}\label{45}
    \|\theta(t)\|_{\dot H^{-1}} \le C \|\theta_0\|_{\dot H^1}\e^{-\lambda t},
\end{equation}
with $\lambda>0$ independent of $\kappa$ and $C$ an almost surely bounded random variable satisfying $\EX_\nu[C^q]<\infty$ for any $q>0$.  A standard stability estimate yields that solutions to the advection-diffusion equation \eqref{eq:AD} satisfy
\[
\frac{1}{2}\frac{\dd}{\dd t}\|\theta(t)\|_{L^2}^2 + \kappa\int_{\T^2}|\nabla\theta(t,x)|^2\dd x = 0.
\]
This implies, on the one hand that the mapping $t\mapsto \|\theta(t)\|_{L^2}$ is monotone decreasing, and on the other hand that for any $\tau>0$,
\begin{equation}\label{451}
2\kappa\int_0^\tau\|\nabla\theta(t)\|_{L^2}^2\dd t \leq \|\theta_0\|_{L^2}^2.
\end{equation}
Therefore, by the mean value theorem, there exists $t'\in (0,\tau)$ such that
\begin{equation}\label{46}
\|\nabla\theta(t')\|_{L^2}^2 \leq \frac{1}{2\kappa\tau}\|\theta_0\|_{L^2}^2.
\end{equation}
Now, if we interpolate between $\dot H^{-1}$ and $\dot H^1$, use the mixing estimate \eqref{45}, and then interpolate between $L^2$ and $\dot H^1$, we can write for any $t\in (2\tau,3\tau)$
\[
\begin{split}
\|\theta(t)\|_{L^2}^2 & \le \|\theta(t)\|_{\dot H^{-1}}\|\grad \theta(t)\|_{L^2} \\
&\le C \e^{-\lambda(t-t')}\|\grad \theta(t')\|_{ L^2}\|\nabla\theta(t)\|_{L^2}.
\end{split}
\]
We then use the monotonicity of $t\mapsto \|\theta(t)\|_{L^2}$ and the previous bound \eqref{46}, and get
\[
\|\theta(t)\|_{L^2}^2 \le C \e^{- \lambda(t-t')}\|\theta_0\|_{L^2} \left(\frac1{2\kappa \tau}\right)^{1/2}  \|\nabla\theta(t)\|_{L^2}  .\]
Integrating over $t\in(2\tau,3\tau)$ and using the monotonicity of the solution in $L^2$ once more, we find that
\[
\tau\|\theta(3\tau)\|_{L^2}^2 \leq   C  \e^{- \lambda\tau }\|\theta_0\|_{L^2} \left(\frac1{2\kappa \tau}\right)^{1/2} \int_{2\tau}^{3\tau} \|\nabla\theta(t)\|_{L^2}  \dd t,
\]
but the integral in the right-hand side, via H\"older's inequality and using \eqref{451}, can be bounded by
\[
\int_{2\tau}^{3\tau}\|\nabla\theta(t)\|_{L^2} \dd t \leq \tau^{1/2}\left(\int_{0}^{3\tau}\|\nabla\theta(t)\|_{L^2}^2\dd t\right)^{1/2} \leq \left(\frac{\tau }{2\kappa}\right)^{1/2} \|\theta_0\|_{L^2}.
\]
So that putting everything together, we obtain
\[
\|\theta(3\tau)\|_{L^2}^2 \le \frac{C}{2\kappa \tau }   \e^{- \lambda\tau}\|\theta_0\|_{L^2}^2.
\]
Notice that we can get rid of $\tau$ in the denominator because there is no degeneracy for $\tau\to 0$ since $\|\theta(\tau )\|_{L^2}\leq \|\theta_0\|_{L^2}$ for all $\tau \geq 0$. Hence, 
 setting $\lambda'=\lambda/6$ and choosing $C'\sim \sqrt{C}$ appropriately, we get the desired estimate with uniform rate.

For small times, $t\lesssim \log\frac1{\kappa}$, this estimate is worse than the trivial monotonicity estimate $\|\theta(t)\|_{L^2}\le \|\theta_0\|_{L^2}$.    In fact, there exists a critical value 
\[
t_* = \frac1{\lambda'}\left(\log\frac{C'}{\sqrt{\kappa}}+1\right),
\]
so that
\[
\|\theta(t_*)\|_{L^2}\le \frac{1}{\mathrm{e}} \|\theta_0\|_{L^2}.
\]
By iterating this estimate, and using the monotonicity in time of the $L^2$ norm of $\theta$, we obtain for any $t\ge0$ and $n\in\N_0$ such that $t\in[nt_*,(n+1)t_*)$,
\[
\|\theta(t)\|_{L^2} \le \|\theta(nt_*)\|_{L^2} \le \e^{-n}\|\theta_0\|_{L^2} \le \e^{-\mu(\kappa)  t}\|\theta_0\|_{L^2},
\]
where $\mu(\kappa)$ is given by
\[
\mu(\kappa) = \lambda'\left(\log\frac{C'}{\sqrt{\kappa}}+1\right)^{-1}.
\]
The estimate and the asymptotic formula for $\mu(\kappa)$  follow immediately.

We finally remark that if $C$ is a random variable, the finiteness of the expectations can be readily checked. 
\end{proof}

\appendix

\section{Feynman--Kac formula}\label{s:feyman-kac}

In this appendix, we will show that the function 
\[
f(t,x,z,y) = \EX_\nu \left[\EX_\kappa \theta_0(X_t^{-1}(x;\{Y_s(y)\}_{s\in[0,t]}))\EX_\kappa \theta_0(Z_t^{-1}(z;\{Y_s(y)\}_{s\in[0,t]}))\rho_0(Y_t^{-1}(y))\right]
\]
defined in Section \ref{s:strategy}, as well as in the proof of Lemma \ref{lemma41}, is a solution to the PDE
\[
    \partial_tf + v_c(x,y)\cdot\nabla_xf + v_c(z,y)\cdot\nabla_zf = \kappa(\Delta_xf + \Delta_zf) + \nu\Delta_yf, 
\]
where $(x,z,y)\in \T^2\times\T^2\times\T^2$. We note that because $X_0$, $Z_0$ and $Y_0$ are all identity maps, and thus, in particular, deterministic, it holds that $f_0(x,z,y) = \theta_0(x)\theta_0(z)\rho_0(y)$.

We start by observing that we can combine the dynamics for $X_t$, $Z_t$ and $Y_t$ given by \eqref{flow} and \eqref{heat} into a single It\^o drift-diffusion process,
\begin{equation}\label{itodrift}
\dd \chi_t = U(\chi_t)\dd t + \sigma dW_t, \quad \chi_0 = \textrm{id}_{\T^6},
\end{equation}
where $W_t =(\tilde B_t,\hat B_t, B_t)$ can now be considered as a standard Brownian motion in $\T^6$, $\sigma\in \R^{6\times 6}$ is a diagonal matrix of the form
\[
\sigma = \sqrt{2}\diag\left(\sqrt{\kappa},\sqrt{\kappa},\sqrt{\kappa},\sqrt{\kappa},\sqrt{\nu},\sqrt{\nu}\right),
\]
and $U:\T^6\to\R^6$ is the vector field defined by
\[
U(x,z,y) = \begin{pmatrix}
    v_c(x,y)\\
    v_c(z,y)\\
    0
\end{pmatrix},
\]
with $v_c$ denoting the $2$-dimensional vector field defined in \eqref{31} and $(x,z,y)\in\T^6$.
To be consistent with the notation introduced previously, here we write $\EX_\sigma$ for the expectation with respect to the noise induced by $W_t$, which in turn  is nothing but integration with respect to the probability measure $\Prob_\kappa\otimes\Prob_\kappa\otimes\Prob_\nu$.

With this representation of the problem it is clear, see for instance the standard reference \cite{Karatzas98}, that the It\^o drift-diffusion process \eqref{itodrift} produces a solution to the Fokker-Planck equation
\begin{equation}\label{fokker-planck}
\partial_t f + U\cdot\nabla f = \frac{1}{2}\sigma^2:\nabla^2 f, \quad f(0,\cdot) = f_0,
\end{equation}
via Feynman--Kac formula
\[
f(t,\xi) = \EX_\sigma f_0(\chi_t^{-1}(\xi)).
\]
On the one hand, it is easy to see that \eqref{fokker-planck} coincides exactly with the two-point PDE \eqref{hypo} because
\[
U\cdot\nabla  = v(x,y)\cdot \nabla_x + v(z,y)\cdot\nabla_z,
\]
and
\[
\frac{1}{2}\sigma^2:\nabla^2 = \kappa\Delta_x + \kappa\Delta_z + \nu\Delta_y.
\]
On the other hand, notice that we can write
\[
\begin{split}
\EX_\sigma f_0(\chi_t^{-1}(\xi)) & = \EX_\sigma \left[(\chi_t)_\#f_0(\xi)\right] = \EX_\sigma\left[(X_t,Z_t,Y_t)_\#f_0(x,z,y)\right] \\
& = \EX_\sigma\left[f_0\left(X_t^{-1}(x),Z_t^{-1}(z),Y_t^{-1}(y)\right)\right]\\
& = \EX_\sigma\left[\theta_0(X_t^{-1}(x))\theta_0(Z_t^{-1}(z))\rho_0(Y_t^{-1}(y))\right]
\end{split}
\]
where we used the subscript $\#$ to denote the push-forward measure. As we mentioned before, $\EX_\sigma$ represents an integration with respect to the probability measure $\Prob_\kappa\otimes\Prob_\kappa\otimes\Prob_\nu$. Although we are slightly abusing the notation, we can identify that the first component of the product probability measure $\Prob_\kappa$ interacts exclusively with $X_t^{-1}(x)$, the second component $\Prob_\kappa$ does the same exclusively with $Z_t^{-1}(z)$, whereas $\Prob_\nu$ interacts with the three processes: with $Y_t$ through diffusion, and with both $X_t$ and $Z_t$ through advection. Thus, this observation yields
\[
\EX_\sigma f_0(\chi_t^{-1}(\xi)) = \EX_\nu \left[\EX_\kappa \theta_0(X_t^{-1}(x))\EX_\kappa \theta_0(Z_t^{-1}(z))\rho_0(Y_t^{-1}(y))\right],
\]
so that we obtain the Feynman--Kac formula for the Fokker--Planck equation \eqref{hypo} and the It\^o process \eqref{itodrift} that we were seeking.

\section*{Acknowledgements}
This work is partially funded by the Deutsche Forschungsgemeinschaft (DFG, German Research Foundation) under Germany's Excellence Strategy EXC 2044--390685587, Mathematics M\"unster: Dynamics Geometry Structure, and  by grant 432402380. VNF gratefully acknowledges support by the ERC/EPSRC Horizon Europe Guarantee EP/X020886/1. The authors thank Andr\'e Schlichting and Samuel Punshon-Smith for stimulating discussions.

\bibliographystyle{abbrv}
\bibliography{euler.bib}

\end{document}